\documentclass[reqno, english,12pt]{smfart}
\RequirePackage{mathrsfs} \let\mathcal\mathscr

\usepackage{a4wide}

\usepackage{color}

\usepackage[colorinlistoftodos]{todonotes}

\usepackage{pifont}

\usepackage{tikz}

\usepackage{amssymb,amsmath,amsfonts, amsthm}

\usepackage{hyperref}
\usepackage{dsfont}
\usepackage{upgreek}
\usepackage{mathrsfs}
\usepackage{mathabx,yfonts}
\usepackage{enumerate}

\usepackage{enumitem}

\usepackage{graphicx}

\newtheorem{theorem}{Theorem}

\newtheorem{lemma}[theorem]{Lemma}
\newtheorem{proposition}[theorem]{Proposition}

\theoremstyle{definition}

\newtheorem*{acknowledgements}{Acknowledgements}
\newtheorem*{notation}{Notation}

\numberwithin{theorem}{section}
\numberwithin{equation}{section}

\DeclareSymbolFont{bbold}{U}{bbold}{m}{n}
\DeclareSymbolFontAlphabet{\mathbbold}{bbold}

\DeclareMathOperator{\vol}{vol}

\newcommand{\md}[1]{  \left(\textnormal{mod}\ #1\right)}
\newcommand{\lab}{\label} 
\renewcommand{\P}{\mathbb{P}}

\newcommand{\Q}{\mathbb{Q}}
\newcommand{\F}{\mathbb{F}}
\newcommand{\N}{\mathbb{N}}
\newcommand{\R}{\mathbb{R}}

\newcommand{\Z}{\mathbb{Z}}
\newcommand{\Zp}{\mathbb{Z}_{\text{prim}}}

\renewcommand{\l}{\left}
\newcommand{\m}{\mathfrak{m}}
\renewcommand{\r}{\right}
\renewcommand{\b}{\mathbf} 
\newcommand{\bb}{\boldsymbol}
\renewcommand{\c}{\mathcal} 
\renewcommand{\gcd}{\textrm{gcd}} 
\renewcommand{\leq}{\leqslant}
\renewcommand{\geq}{\geqslant}
\renewcommand{\#}{\sharp}
\renewcommand{\gg}{\ggg}

\renewcommand{\ll}{\lll}

\newcommand{\p}{\mathfrak{p}}

\DeclareMathOperator*{\Osum}{\sum{}^*}

\newcommand{\beq}[2]
{
\begin{equation}
\label{#1}
{#2}
\end{equation}
}

\begin{document}

\date{}

\title[
Density of
isotropic fibres in conic bundles
]{
Serre's problem on
the density of
isotropic fibres
in conic bundles
}
 
\author{E. Sofos}

\subjclass{$14$G$05$; $14$D$10$, $11$N$36$, $11$G$35$}

\address{
Universiteit Leiden, 
Mathematisch Instituut Leiden,
Snellius building, Niels Bohrweg 1,
2333 CA Leiden, Netherlands.}
\email{e.sofos@math.leidenuniv.nl}

\begin{abstract}
Let $\pi:X\to \P^1_{\Q}$ be a non-singular conic bundle over $\Q$ 
having $n$ 
non-split fibres
and
denote by
$N(\pi,B)$
the cardinality of the
fibres of Weil height at most
$B$
that possess a rational point.
Serre showed
in $1990$
that a direct application of the large sieve
yields
\[
N(\pi,B)
\ll B^2(\log B)^{-n/2}
\]
and
raised the problem
of proving that this is
the true order of magnitude of $N(\pi,B)$
under the necessary assumption that there exists
at least one smooth fibre with a rational point.
We solve this problem for all non-singular conic bundles of
rank at most $3$.
Our method comprises the use of Hooley neutralisers,
estimating divisor sums over values of binary forms,
and an application of the Rosser--Iwaniec sieve.
\end{abstract}

\maketitle

\setcounter{tocdepth}{1}
\tableofcontents

\section{Introduction}
\lab{detect} 

The asymptotic distribution of the members of a family of varieties that have a rational point
has been the focus of intensive investigation during recent years.
There are families of
Fano varieties over $\Q$
where
the percentage of fibres
with a $\Q$-point exists and is positive.
Poonen and Voloch~\cite{poonen_voloch} have verified 
this in the case of hypersurfaces on the condition that the Brauer-Manin obstruction is the
only obstruction to the Hasse principle for rational points
on certain varieties.
Research on this theme has subsequently flourished; 
the interested
reader is
referred to the recent work
of Loughran and Smeets~\cite{smeets}
and the thorough list of references it provides.
It is noteworthy that 
families of conics were
excluded from
the Poonen-Voloch results 
since in this case the percentage is expected to vanish.
Our aim in this paper is to focus on this exceptional case.

The exceptional behavior,
first noticed by Serre~\cite{Ser90},
can be explained, for example,
through~\cite[Th.1.2]{bharg} of
Bhargava, Cremona, Fisher, Jones,
and Keating, where
it is stated that
the
probability that a quadratic form in 
$3$
variables over
$\Z_p$ 
is isotropic
is given by
\[ \rho_3(p)=1-\frac{p}{2(p+1)^2}.\]
Owing to the Hasse principle,
this suggests that the percentage of 
isotropic planar conics over $\Q$ with coefficients of 
size $B$
should vanish 
asymptotically
as $B\to \infty$
and, 
more precisely,
it should decrease like
\[\prod_{p\leq B} \rho_3(p)
\asymp (\log B)^{-1/2}.\]

Serre~\cite{Ser90} used the large sieve
to
prove
upper bounds
of the aforementioned order of magnitude
and
raised the problem of verifying that
this is
the correct order of magnitude.
Obtaining
precise lower bounds
is a genuinely harder problem,
since one needs a method for finding rational points.
There are fewer difficulties when the generic fibre satisfies the Hasse principle,
however, even in this case,
there have
been only very special cases in which the problem has been solved.
More specifically,
Hooley~\cite{hool1}
and
Guo~\cite{guo}
proved the correct lower bound for the case of diagonal planar conics,
and later,
Hooley~\cite{hool2}
proved a
similar
result in the case of general planar conics.

Let $\pi:X\to \P^1_{\Q}$ be a non-singular conic bundle over $\Q$.
A singular fibre $\pi^{-1}(\m)$
above a closed point  $\m \in  \P^1_{\Q}$
is called \emph{split} if
both of its components are
defined over
the residue field $\Q(\m)$
and \emph{non-split} otherwise.
Denoting the set of non-split fibres 
by $M(\pi)$,
and letting \[n=\sharp M(\pi),\]
Serre's
problem
reads as follows:\textit{
assuming that
there exists a smooth 
$\Q$-isotropic fibre of $\pi$,
then verify that
the quantity
\[
N(\pi,B)
:=\#\Big\{x \in \P^1_{\Q}:
H(x)\leq B, 
\pi^{-1}(x) \ \text{has a rational point} 
\Big\}
\]
satisfies 
\beq{intro ser}
{
N(\pi,B)
\asymp \frac{B^2}{(\log B)^{n/2}},
}
where $H$ is
the usual Weil height on $\P^1_{\Q}$.}
It should be noted that Loughran~\cite{Lou13}
has stated a broader and more precise
form of~\eqref{intro ser}
that he
furthermore verified in specific situations.
Following Skorobogatov~\cite{skoratsis}, we let
\[
r(\pi)=\sum_{\m \in M(\pi)}
[\Q(\m):\Q]
\]
denote the \emph{rank} of the fibration $\pi$.
As a special case of~\cite[Cor.1.4]{Lou13}
the estimate~\eqref{intro ser} is promoted to an asymptotic formula
when $r(\pi)=2$ and $n=1$,
and
the proof
uses harmonic analysis
on toric varieties.  
Our main theorem settles Serre's problem in cases 
where one lacks such a structure.
\vskip 1cm 
\begin{theorem}
\label{thm:lower}
Let $\pi:X\to \P^1_{\Q}$ be a
non-singular conic bundle over $\Q$ of rank
$r(\pi)\leq 3$
and assume that there exists a smooth fibre with a rational point. 
Then
\[
N(\pi,B)
\asymp
\frac{B^2}{(\log B)^{n/2}}.\]
\end{theorem}
\vskip 0.3cm 
The existence assumption is obviously necessary for the lower bound.
However, in the case $r(\pi)=3$  
it becomes redundant.
Indeed, when $r(\pi)=3$, 
$X$ becomes
birational 
to a quintic del Pezzo surface
by the
work of Iskovskih~\cite[Th.5]{isko}
and therefore the theorem of Enriques~\cite{enri}
combined with~\cite[Th.29.4]{maninakos}
reveals that $X$ is $\Q$-rational,
and in particular it has at least one smooth fibre with a $\Q$-point.

In what follows
we outline the proof of Theorem~\ref{thm:lower}
as well as the structure of the paper.
Constructing
indicator functions that detect the isotropic fibres
is
a key
novel ingredient of the
proof of Theorem~\ref{thm:lower}.
It allows the transformation of
$N(\pi,B)$
into an average of
arithmetic functions over values of binary forms
(see~\eqref{15} and~\eqref{eq:explicit formula}).~The construction
will be brought to life
in~\S\ref{Aat}.

The next step takes place in~\S\ref{1}
and regards the most important technical device of our paper,
namely \textbf{Hooley neutralisers}.
Using the approach expounded by Hooley~\cite{hool3} 
allows us to extirpate certain awkward
arithmetic functions
apparent
in~\eqref{15},
by introducing new weights coming from a combinatorial sieve.
This reduces the problem to one of estimating
asymptotically 
specific divisor sums,
with an error term that exhibits a power saving.

A precise definition of 
the divisor sums at hand
is supplied
in~\eqref{s d sum}.
They are of 
shape
\[\sum_{|s|,|t|\leq B}
\
\prod_{i=1}^n
\
 \sum_{d_i|\Delta_i(s,t)}\l(\frac{F_i(s,t)}{d_i}\r)
,\]
where
$(\frac{\cdot}{\cdot})$
denotes the Jacobi symbol,
$F_i \in \Z[s,t]$ are forms
of even degree,
$\Delta_i \in \Z[s,t]$ are irreducible forms with 
$\sum_i\deg(\Delta_i)\leq 3$,
and the summation is over
coprime integers $s$ and $t$,
satisfying certain
congruence conditions.
There is a
large volume of literature
that builds on the seminal work
of S. Daniel~\cite{stephan}
to evaluate similar averages.
However,
his approach allows a power saving in the error term
only
in the case
$\sum_i\deg(\Delta_i) \leq 3$,
which is the only element in our proof 
of Theorem~\ref{thm:lower} obstructing
its generalisation to all non--singular conic bundles with a smooth $\Q$-isotropic fibre.
We shall state our result regarding divisor sums
in Theorem~\ref{b evans}
and its proof
will be
the sole aim in~\S\ref{stephan d}.

A confluence of entirely new
features appear in~\S\ref{stephan d}
for the first time in the context of divisor sums over values of binary forms.
Initially,
one is faced with the presence of the forms $F_i$ rather than fixed integers.
This poses added difficulties 
that require serious modification of Daniel's approach
(see Proposition~\ref{let it be}). 
A further novel
feature of our treatment
lies in the fact that our method
is uniform
in the factorisation of the
form $\prod_i \Delta_i$ over $\Q$
(see Lemma~\ref{lemmon}).
Lastly, we will prove that the modulus defining the aforementioned congruence conditions on $s$ and $t$
is allowed to assume values up to a positive power of $B$.
This \textit{level of distribution} result is the most advanced new
component of our work; we shall
put it
into form in Theorem~\ref{b evans}.

The Rosser-Iwaniec sieve
enters the stage
during~\S\ref{combo mac iwa}
and it is the final step in the proof of Theorem~\ref{thm:lower}.
As is surely
familiar to sieve experts, 
the use of sieve weights precludes the possibility of obtaining an asymptotic 
for $N(\pi,B)$ in Theorem~\ref{thm:lower}.
\begin{acknowledgements}
The problem was suggested during the  $2014$ 
AIM Workshop 
\emph{`Rational and integral points on higher-dimensional varieties'}
by Daniel Loughran
to whom
we are grateful for numerous conversations. 
We are furthermore indebted to Roger Heath--Brown for
valuable discussions.
This investigation was performed while the author was a postdoctoral researcher at
Leiden and the support of its university is greatly appreciated.
\end{acknowledgements} 
\begin{notation}
For any $k,k' \in \Z$ we shall denote their
greatest common divisor and  
least common multiple by
$\gcd(k,k')$ and $[k,k']$ respectively.
The set of all primitive integer vectors in $\Z^k$
will be denoted by
$\Zp^k$, while the $p$-adic valuation of an integer $k$
will be denoted by $\nu_p(k)$.
As usual, we let
$\tau(k)$
denote the number of positive divisors of any non-zero
integer $k$. Each
$\b{x} \in \R^k$ 
has a supremum norm which will be denoted by
$\|\b{x}\|$,
and similarly,
for a bounded set $\c{C}\subset\R^k$
we shall represent the supremum of
$\Big\{\|\b{x}\|:\b{x}\in \c{C}\Big\}$
by
$\|\c{C}\|_\infty$.
Given a polynomial
$f\in \R[x_1,\ldots,x_k]$ we shall denote the maximum absolute value of its coefficients by
$\langle f \rangle$.

The data associated to the conic bundle apparent in
Theorem~\ref{thm:lower} will be considered constant throughout.
This is taken to mean that,
although each implied constant
in the big $O$
notation
will depend on several quantities related
to $\pi$,
we shall avoid recording these dependencies.
The list of the said quantities 
consists of 
\[
\Delta,\Delta_i,F_i,\c{R},W,W_i,f_i \in \c{U},s_0,t_0,
\]
whose meaning
will become evident in due course.
Any other dependencies of the implied constants on further parameters 
will be explicitly specified via the use of a subscript.
The symbol
$\varepsilon$ will be used for a small positive parameter
and its value may vary, allowing,
for example, inequalities of the form
$x^\varepsilon \ll_\varepsilon x^{\varepsilon/2}$.
\end{notation}

\section{Preliminaries}
\lab{breath}
We shall follow closely the geometric setup 
for conic bundles as 
given in~\cite[\S 2.2]{brick}.
The fibre of 
$\pi$
above the point $[s,t] \in \P^1_\Q$
is of the shape
$Q_{s,t}=0$, where
\beq{sinister}{
Q_{s,t}(\b{x})=\sum_{i,j=1}^3
f_{ij}(s,t)x_ix_j
}
and all
$f_{ij}$
are binary
integer forms.
The
non-singularity of
$X$
shows that
the discriminant of $Q_{s,t}$,
defined through
\[
\Delta(s,t)=\frac{1}{2}\det\l(\frac{\partial^2 Q_{s,t}(\b{x})}{\partial x_i\partial x_j }\r)
,\]
is
separable.
We shall assume that all 
principal minors 
of the
matrix $(f_{ij})\in (\Z[s,t])^{3 \times 3}$
are forms of 
even degree, say $d_i \in 2\Z_{\geq 0}$
for $i=1,2,3$.

The fact that we are
interested in lower bounds 
of the correct order of magnitude
for $N(\pi,B)$
allows us to assume with no loss of generality, that the fibre at infinity is smooth,
a fact equivalent to $t\nmid \Delta(s,t)$ in $\Z[s,t]$.
It furthermore allows
us to assume that
all singular fibres are non-split,
since contracting one
line in any
split singular 
fibre does not affect the order of magnitude of $N(\pi,B)$.
We have the factorisation
\beq{factori}
{\Delta(s,t)=\prod_{i=1}^n \Delta_i(s,t),
}
where the forms $\Delta_i$ are irreducible
over $\Z[s,t]$
and coprime in pairs;
they are in correspondence with the 
singular fibres of $\pi$.
For each $i=1,\ldots,n$
we fix
a root $\theta_i \in \overline{\Q}$
of $\Delta_i(x,1)=0$
once and for all for the rest of this paper. The fact that each singular fibre is non-split
implies that non of the lines comprising the singular fibre above
$[\theta_i,1]$ is defined over the residue field
$\Q(\theta_i)$.

Let us
recall
that in the notation of~\cite[\S 2.2]{brick},
the
morphism 
\[(s,t;\b{x})
\mapsto
(s/t;[x_1t^{-d_1/2},x_2t^{-d_2/2},x_3t^{-d_3/2}])
\]
maps each $(s,t;\b{x})\in \mathbb{A}^5$
with 
$Q_{s,t}(\b{x})=0,\b{x} \neq \b{0}$
and
$(s,t)\neq \b{0}$
to
a point 
on the variety
$S_1\subset \mathbb{A}^1\times  \P^2$ 
given by
$Q_{u,1}(\b{x})=0$. 
This reveals that there exists an integer $d$
that makes the following
identity
valid, \beq{lamia}{
t^{d}Q_{s,t}(\b{x})=Q_{s/t,1}(x_1t^{-d_1/2},x_2t^{-d_2/2},x_3t^{-d_3/2}).}
\section{Construction of the detectors}
\lab{Aat}
Owing to
the assumptions of Theorem~\ref{thm:lower}
there exist $(s_0,t_0) \in \Zp^2$ and
$\b{y}\in \Zp^3$ such that 
$Q_{s_0,t_0}(\b{y})=0$ and $\Delta(s_0,t_0)\neq 0$.
The implicit function theorem for the $\R$-morphism 
$\pi:X(\R)\to \P^1(\R)$ provides a set
$C\subset \R^2$ 
of non-empty interior that 
contains both
$\b{0},(s_0,t_0)$
and
that is closed under scalar multiplication, such that 
whenever $(s,t) \in C$ then $Q_{s,t}=0$ is smooth and has a real point.
Observe that given any finite set of lines of $\R^2$
all of which pass through
the origin but none through $(s_0,t_0)$,
there exists a box inside $C$ that contains
 $(s_0,t_0)$
and not intersecting 
any of the lines.
This shows that there exists a box of non-empty interior $\c{R}\subset C$ 
such that if
$(s,t) \in \R^2\cap B\c{R}$ and $B\in \R_{>1}$ then 
$
\Delta(s,t)\neq 0,
Q_{s,t}(\R)\neq \emptyset
$
and
\beq{labyr}{
\Delta_i(s,t)
\Delta_i(s_0,t_0)>0
\ \text{if} \ 
\Delta_i
\ \text{is linear}
.}
In a similar fashion, the
non-archimedean analogue of
the implicit function theorem 
can be used
for each $p\leq D_0$
to
provide large integers $\kappa_p$
such that 
\beq{wellser}{
\text{if}
\
(s,t)
\equiv 
(s_0,t_0)
\md{p^{\kappa_p}}
\ \text{then} \ 
Q_{s,t}(\Q_p)\neq \emptyset
.}

We define the symbol
$\c{H}_p(s,t)$
for
all primes $p$ 
and any $(s,t) \in \Zp^2$ with $\Delta(s,t)\neq 0$ to be
$1$ or $-1$ according to if 
the conic $Q_{s,t}(\b{x})=0$ is soluble over $\Q_p$ or not.
Therefore
letting
for any $B>1$,
\[
\c{B}:=
\Bigg\{\left(s,t\right) \in \Zp^2
\cap B\mathcal{R}:
(s,t)\equiv (s_0,t_0)\md{\prod_{p\leq D_0}p^{\kappa_p}}
\Bigg\},\]
allows us to deduce via the Hasse principle
that
\beq{for any}{
N(\pi,B)
\gg
\sum_{\substack{(s,t) \in \c{B}}}
\prod_{\substack{p|\Delta(s,t)\\p>D_0}}
\l(\frac{1+\c{H}_p(s,t)}{2}\r),
}
given that
$Q_{s,t}(\b{x})=0$
has a $\Q_p$-point
whenever
$p\nmid \Delta(s,t)$.

It is important to highlight that~\eqref{for any}
gives us the freedom
to choose $D_0$ arbitrarily large.
The idea of restricting to congruences modulo the integer
\[W=\prod_{p\leq D_0}p^{n_p},\]
where $n_p\geq \kappa_p$,
comes from sieve theory. We shall
often enlarge the size of 
$D_0$ and $n_p$ with no further mention,
this will allow us to deal with complications arising from the effect of small primes.
We shall always choose $W$ to depend solely on the coefficients of the 
forms $f_{ij}(s,t)$ in~\eqref{sinister} and in particular $W$ will   
be independent of the counting parameter $B$.

Owing to the vanishing of
$\Delta(\theta_i,1)$,
the
form $Q_{\theta_i,1}(\b{x})$
is singular and thus
$\nabla Q_{\theta_i,1}(\b{v})=0$
for some
non-zero vector
$\b{v}\in \Z[\theta_i]^3$.
Fixing
some
$1\le \ell_i \le 3$
satisfying $v_{\ell_i}\neq 0$ allows us to consider the discriminant of the binary quadratic form
$Q_{s,t}(x_{\ell_i}=0)$. This discriminant, henceforth called $F_i(s,t)$,
is a binary form in $\Z[s,t]$
and has 
even degree owing to the assumption that 
the matrix of the quadratic form
$Q_{s,t}(\b{x})$ has principal minors of even degree. 
It is important to note that the construction of $F_i$ is not unique
since it depends on the choice of ${\ell_i}$.
We can however make a choice of ${\ell_i}$ following
the 
algorithm above,
and fix this choice once and for all
for the rest of this paper.

For any
$N \in \N$ and any field $K$ of characteristic not $2$,
we let $Q \in K[x_1,\ldots,x_N]$ be any quadratic form.
Then for all $\b{x} \in K^N$, we have
$
4Q(\b{x})=
2Q(\b{x})+\nabla Q(\b{x}) \cdot \b{x} 
$,
where $\cdot$ denotes the inner product in $K^N$.
Therefore we obtain that 
\beq{eq:nab}
{
\text{if}
\
\
\nabla Q(\b{x})=0,
\
\
\text{then}
\
\
Q(\b{x})=0
.} 

Let us next prove that the fact that the singular fibre above
$[\theta_i,1]$ is not a double line implies $\mathrm{Res}(\Delta_i,F_i)\neq 0$.
Indeed, assuming with no loss of generality 
that $\ell_i=1$
and considering
the transformation $\b{y}=T_i^{-1}\b{x}$, where
\[
 T_i:=
\left( \begin{array}{ccc}
v_1 & 0 & 0 \\
v_2 & v_1 & 0 \\
v_3 & 0 & v_1 \end{array} 
\right),
\]
shows that 
\begin{align*}
Q_{\theta_{i},1}(\b{x})
&= Q_{\theta_i,1}(v_1(0,y_2,y_3)+y_1 \b{v})
\\
&=  v_1^2 Q_{\theta_i,1}(0,y_2,y_3)
+y_1^2 Q_{\theta_i,1}(\b{v})
+y_1 v_1 (\nabla Q_{\theta_i,1}(\b{v}) \cdot (0,y_2,y_3))
\\
&=v_1^2
Q_{\theta_i,1}(0,y_2,y_3).
\end{align*}
We have made use of 
$Q_{\theta_i,1}(\b{v})=0$,
which is implied by 
$\nabla Q_{\theta_i,1}(\b{v})=0$
and of~\eqref{eq:nab}
for 
\[
N=3,
K=\Q(\theta_i),
Q=Q_{\theta_i,1},
\b{x}=\b{v}
.\]
The calculation above reveals that 
there exists an invertible
transformation
defined over $\Q(\theta_i)$
that transforms the singular fibre above
$[\theta_i,1]$ 
into the binary quadratic form
\[Q_{\theta_i,1}(0,y_2,y_3)=0.\]
Since our conic bundle is non-singular,
each singular fibre is not a double line 
and therefore the discriminant of the said quadratic form, $F_i(\theta_i,1)$,
must be non-zero. However, if
$\mathrm{Res}(\Delta_i,F_i)=0$
then the irreducibility of $\Delta_i$ would yield
that $\Delta_i|F_i$ in $\Z[s,t]$
and therefore $F_i(\theta_i,1)=0$,
a contradiction.
Let us record here the obvious observation that
since
the singular
fibre above $[\theta_i,1]$
is
non-split, we have that  
\beq{splitt}{F_i(\theta_i,1) \notin \Q(\theta_i)^2
\ \
\text{for all}
\ \
i=1,\ldots,n.}
\begin{proposition}
[Detectors]
\label{soll}
There exists a large positive
constant $D_0$,
that depends at most on the coefficients of the equation 
$Q_{s,t}=0$
which defines the conic bundle,
such that,
whenever
$(s,t) \in \Zp^2$ with $\Delta(s,t)\neq 0$
and $p>D_0$ is a prime that divides $\Delta(s,t)$,
then there exists a unique 
$i \in \{1,\ldots,n\}$
satisfying
$p|\Delta_i(s,t)$
and  
\[
\c{H}_p(s,t)=
\l(\frac{F_i(s,t)}{p}\r)^{\!\nu_p(\Delta_i(s,t))}
.\]
\end{proposition}
\begin{proof}
Recalling~\eqref{factori}
and taking $D_0$ large enough so that 
$W$ includes all prime divisors of the resultant of $\Delta_i$ and $\Delta_j$
for all $i\neq j$,
immediately yields that there exists a unique index
$i \in \{1,\ldots,n\}$ such that
$p|\Delta_i(s,t)$.
The fact that 
$t\nmid \Delta(s,t)$ in $\Z[s,t]$
implies that 
$\Delta_i(1,0)\neq 0$.
Enlarging $D_0$
allows us to assume that 
$p\nmid \Delta_i(1,0)$
and therefore
the coprimality of $s$ and $t$ reveals that $p\nmid t$.
We therefore obtain that 
$\Xi_p:=s/t\md{p}$
is
defined
and, in light of~\eqref{lamia},
$Q_{s,t}=0$ has a $\Q_p$-point 
if and only if $Q_{\Xi_p,1}=0$ does.
Noting
that 
$\Delta_i(\Xi_p,1)\equiv 0 \md{p}$,
allows us to confirm that 
the map
$\psi:\Z[x]/(\Delta_i(x,1))  \to \Z/p\Z$,
given by
\[
h(x) +(\Delta_i(x,1)) \mapsto h(\Xi_p)\md{p}
,\]
is well-defined and a ring homomorphism.

Taking a larger value for $D_0$ allows us to obtain the coprimality of $p$ and $\psi(v_{\ell_i}) 2\mathrm{Res}(\Delta_i,F_i)$ 
and hence the matrix $\psi(T_i) \in (\Z/p\Z)^{3\times 3}$
is invertible. Therefore applying $\psi$
to the equality
$Q_{\theta_{i},1}(\b{x})=v_{\ell_i}^2
Q_{\theta_i,1}(0,y_2,y_3)
$
shows that 
$Q_{\Xi_{p},1}$
can be transformed over $\Z/p\Z$
into the 
binary quadratic form
$Q_{\Xi_p,1}(0,y_2,y_3).$
Due to 
$p\nmid
\mathrm{Res}(\Delta_i,F_i)$,
its discriminant
$F_i(\Xi_p,1)$,
is not divisible by $p$
and we can therefore 
diagonalise it into
$c(z_1^2-F_i(\Xi_p,1)z_2^2)$
for some $c \in \Z/p\Z$ with $ p\nmid c$.
We have therefore shown that
$Q_{\Xi_p,1}=0$
is equivalent 
over $\Q_p$
to 
\[c(z_1^2-F_i(\Xi_p,1)z_2^2)+c'z_3^2=0\]
for some
$c' \in p\Z_p$.
Noting that the $p$-adic valuation of $\Delta(\Xi_p,1)$
is $\nu_p(\Delta_i(s,t))$, we deduce that 
$\nu_p(c')=\nu_p(\Delta_i(s,t))$
and therefore
a standard computation reveals that the $p$-adic Hilbert symbol of 
$z_1^2=F_i(\Xi_p,1)z_2^2-\frac{c'}{c}z_3^2$
equals
\[
\l(\frac{F_i(\Xi_p,1)}{p}\r)^{\nu_p(\Delta_i(s,t))}
.\] 
Bringing together the observations
$s\equiv \Xi_p t \md{p}$
and $2|\deg(F_i)$
finishes the proof of our proposition.
\end{proof}
Introducing for each $(s,t) \in \c{B}$
and $i \in \{1,\ldots,n\}$
the entities
\[
r_i^*(s,t):=
\prod_{\substack{p|\Delta_i(s,t)\\p>D_0}}\l(1+\l(\frac{F_i(s,t)}{p}\r)\r)
\]
and
combining Proposition~\ref{soll}
with~\eqref{for any}
yields that 
$N(\pi,B)$ is
\[\gg \sum_{(s,t) \in \c{B}}
\prod_{i=1}^{n}\prod_{\substack{p|\Delta_i(s,t)\\p>D_0}}
\frac{1}{2}\l(1+\l(\frac{F_i(s,t)}{p}\r)^{\nu_p(\Delta_i(s,t))}\r)
\geq 
\sum_{(s,t) \in \c{B}}
\prod_{i=1}^{n}
r_i^*(s,t)
2^{-\#\{p|\Delta_i(s,t):p>D_0\}}
.\]
Define for any
$k \in \Z-\{0\}$
and
$z\in \R\cap (D_0,\infty)$
the function
\[\omega(k;z):=\#\{p|k:D_0<p\leq z\}.\]
Furthermore,
whenever
$z\in \R\cap (D_0,B)$
we define
\[\varpi:=\l(\frac{\log z}{\log B}\r)100n(n+1),
\]
which will eventually be
a small positive constant.
Let us recall that for the integer
$j=\#\{p|k:p>D_0\}$
we have
$j\leq (\log k)/(\log D_0)$.
The
inequality 
$|\Delta_i(s,t)|\ll B^{\deg(\Delta_i)}$
therefore reveals that  
\begin{equation}
\lab{15}
N(\pi,B)
\gg
\sum_{(s,t) \in \c{B}}
\
\prod_{i=1}^n
\frac{r_i^*(s,t)}{2^{\omega(\Delta_i(s,t);z)}}
.\end{equation}
Before the end of this section 
we record a result 
which will facilitate the
application of the hyperbola trick in~\S\ref{stephan d}.
For any integer 
$n \in \Z-\{0\}$
define
\[ 
n^\dagger:=\prod_{\substack{p|n\\p>D_0}}p^{\nu_p(n)}
.\]
\begin{proposition} 
\lab{let it be}
For all
$(s,t) \in \c{B}$
and
$i=1,\ldots,n$
we have
\[
\l(\frac{F_i(s,t)}{\Delta_i(s,t)^\dagger}\r)=1
.\]
\end{proposition}
\begin{proof}
Notice that for all $(s,t)$ under consideration
the conic
$Q_{s,t}(\b{x})=0$ is smooth
and is soluble over $\R$ and $\Q_p$
for all primes $p$ satisfying
$p\leq D_0$ or $p\nmid \Delta(s,t)$.
By Hilbert's reciprocity formula we obtain
\beq{hilbe0}{\prod_{\substack{p|\Delta(s,t)\\p>D_0}}\c{H}_p(s,t)=1}
and hence Proposition~\ref{soll} yields that 
\[\prod_{i=1}^n
\l(\frac{F_i(s,t)}{\Delta_i(s,t)^\dagger}\r)=1.\]
Noting that the rank $r(\pi)=\sum_{i=1}^n\deg(\Delta_i)$
is at most $3$
shows that the vector $(\deg(\Delta_i))_{i=1}^n$
can only be one of the following, $(3), (2,1), (1,1,1), (2),(1,1)$ or $(1)$.
In the first case,
the proof of our claim
is furnished immediately by~\eqref{hilbe0}, 
while in the two remaining cases,
~\eqref{hilbe0}
shows that 
it suffices to prove
$
\l(\frac{F_i(s,t)}{\Delta_i(s,t)^\dagger}\r)=1$
for each $i$ with $\Delta_i$ linear.

Indeed, let $\Delta_i(s,t)=a_is-b_it$ be such a form, where $a_i,b_i$ are integers with 
$(a_i,b_i)\neq \b{0}$.
Letting
$c_i:=F_i(b_i,a_i)$,
we observe that
there exists a polynomial $g_i \in \Z[x,y]$
such that
\[
a_i^{\deg(F_i)}
F_i(x,y)
=y^{\deg(F_i)}
c_i
+\Delta_i(x,y)
g_i(x,y)
,\]
as can be shown by a Taylor expansion for example.
Specialising to $(x,y)=(s,t)$ 
and using that
$2|\deg(F_i)$,
we obtain the equality
\[
\l(\frac{F_i(s,t)}{\Delta_i(s,t)^\dagger}\r)=
\l(\frac{c_i}{\Delta_i(s,t)^\dagger}\r).\]
To continue our argument, we augment $W$
by assuming that 
$4c_i\Delta_i(s_0,t_0) | W$.
Note that for all $(s,t)$ in our lemma there exist $(s',t')\in \Z^2$
such that
$(s,t)=(s_0,t_0)+W(s',t')$
and hence we have that  
\[\Delta_i(s,t)=\Delta_i(s_0,t_0)\l(1+\frac{W}{\Delta_i(s_0,t_0)}\Delta_i(s',t')\r).\]
The integer in the parenthesis is positive
due to~\eqref{labyr}. Noting that $\Delta_i(s_0,t_0)|W$
shows that
$\Delta_i(s_0,t_0)^\dagger=1$, and therefore
\[\Delta_i(s,t)^\dagger=
1+\frac{W}{\Delta_i(s_0,t_0)}\Delta_i(s',t')
\equiv 1 \md{4c_i}
.\]
Now
quadratic reciprocity reveals
the validity of
\[
\l(\frac{c_i}{\Delta_i(s,t)^\dagger}\r)=
\l(\frac{\Delta_i(s,t)^\dagger}{c_i}\r),\]
which equals $(\frac{1}{c_i})=1$,
and therefore
our proof is hereby 
concluded.
\end{proof}

We shall now record an explicit formula for $N(\pi,B)$ which is valid for non-singular
conic bundles
of any rank.
This formula is not used in the present paper but can be helpful in future work in the area.
Let us define for any index
$i=1,\ldots,n$ and all
$(s,t) \in \Zp^2$ satisfying
$\Delta(s,t)\neq 0$ the arithmetic
functions 
\[
\widetilde{r}_i(s,t):=  
\prod_{\substack{p|\Delta_i(s,t)\\p>D_0}}
\l(1+\c{H}_p(s,t)\r)
,\]
which, by Lemma~\ref{soll}, fulfil
\[
\widetilde{r}_i(s,t)
=
\hspace{-0,3cm}
\sum_{\substack{d_i|\Delta_i(s,t)^\flat \\
\gcd(d_i,\Delta_i(s,t)/d_i)=1}}
\hspace{-0,3cm}
\l(\frac{F_i(s,t)}{d_i}
\r)
.\]
Introducing for any
$D>0$
the set
\[
\mathrm{Sol}(D):=
\{
(s,t) \in \Z^2:
Q_{s,t}=0
\ \text{isotropic over} 
\
\R
\
\text{and}
\
\Q_p
\
\text{for all}
\
p\leq D
\}
,\]
allows us to
deduce  
via the 
Hasse principle that
\beq{eq:explicit formula}{
N(\pi,B)=\
\frac{1}{2}
\Osum_{\substack{|s|,|t| \le B\\ \Delta(s,t) \neq 0 \\
(s,t) \in \mathrm{Sol}(D_0)
}}
\prod_{i=1}^n
\frac{\widetilde{r}_i(s,t)}{2^{\#\{p|\Delta_i(s,t)^\flat\}}}
+
\sum_{\substack{1\le i \le n 
\\
\deg(\Delta_i)=1}}
\hspace{-0,3cm}
1 
,}
since the only degenerate fibres contributing towards $N(\pi,B)$ are the ones defined over $\Q$.

We next
sketch how one could deal with the conditions on the summation over $s,t$
imposed by the set $\mathrm{Sol}(D_0)$. 
We can see, for example, by diagonalising $Q_{s,t}=0$ over $\R$ and using the Hilbert symbol over $\R$,
that there is a finite union of open, disjoint, and non-empty
sets $\c{R}_j \subset \R^2$ such that $Q_{s,t}=0$ has a real point if and only if $(s,t) \in \bigcup_j \c{R}_j$.
To deal with $p$-adic solubility for the small primes $p\leq D_0$, we can observe that for each prime $p\geq 2$,
whenever $p^\alpha\| \Delta(\sigma,\tau)$ and $(s,t) \equiv (\sigma,\tau) \md{p^{\alpha+1}}$, then
$Q_{s,t}=0$ has a $\Q_p$-point if and only if 
$Q_{\sigma,\tau}=0$ has, and therefore one can then partition
the sum over $s,t$ in sets defined $\md{p^{\alpha+1}}$ for all $p\leq D_0$
and $\alpha \in \Z_{\geq 0}$.
The periodic property we mentioned can be inferred by observing that 
a smooth conic $Q=0$ has a $p$-adic point if and only if its
$p$-adic Hardy-Littlewood density is positive and that
the limit~\cite[Eq.(1.2)]{conconcon}
defining the $p$-adic density stabilises as soon as $n\geq \nu_p(\Delta_Q)+1$,
in the notation of~\cite{conconcon}.

\section{Introducing the Hooley neutralisers}
\lab{1}
\subsection{Construction of the neutraliser}
\lab{1a} 
It is tempting to use the formula
\[\frac{1}{2^{\omega(k)}}=\sum_{d|k}
\frac{\mu(d)}{\tau(d)}
\] 
to turn the sum~\eqref{15}
into a sum of the form
\[\sum_{d_i\ll B^{\deg(\Delta_i)}}\frac{\mu(d_1)\cdots \mu(d_n)}{\tau(d_1)\cdots \tau(d_n)}
\sum_{\substack{(s,t) \in \c{B} \\ d_i|\Delta(s,t)}} r_1^*(s,t)\cdots r_n^*(s,t).
\]
Alas,
this action will place us in a quandary
owing to the magnitude of the range of summation of each
variable $d=d_1,\ldots,d_n$.
Hooley's neutraliser trick consists of employing
sieve functions 
$\lambda_d^\pm$
that
imitate the M\"{o}bius function
$\mu(d)$, yet at the same time, having 
a small support.

With an eye to future applications we supply a general version of his artifice in the next proposition.
Before proceeding, we recall the sole
required
property of the sequence 
$(\lambda^{\pm}_d)_{d \in \N}$, namely
\[
\lambda^\pm_1=1
\ \ \text{and} \ \
\lambda^-_d\leq 0 \leq 
\lambda^+_d
\ \ \text{for all} \ d \neq 1.\]
\begin{proposition} 
\lab{12++} 
Assume that 
$\c{P}_1,\ldots, \c{P}_n$
are sets of primes
and
that we are given 
functions $f_1,\ldots,f_n:\N\to \R$  
satisfying the following properties,
\newline
$(1)$
$f_i$ multiplicative,
\newline$(2)$
$f_i(\N)\subset \R_{>0}$, 
\newline$(3)$
$m\in \N, p \ \text{prime} \Rightarrow f_i(p^m)=f_i(p)$, 
\newline$(4)$
$p \in \c{P}_i \Rightarrow f_i(p)<1$, 
\newline$(5)$
$p \notin \c{P}_i \Rightarrow f_i(p)=1$.
\newline
Define the multiplicative functions 
$
\widehat{f_i} 
:\N
\to \R$ by 
$\widehat{f_i}(k):=
\prod_{p|k}
\l(1-f_i(p)\r)$.
Then for all
$\b{k} \in \N^n$ with
$\gcd(k_i,k_j)=1$ for all
$i\neq j$
we have
\[
\sum_{\substack{ 
d_i|k_i  \\ 
p|d_i \Rightarrow p \in \c{P}_i \\ 
d_i \ \text{squarefree}}}
\hspace{-0,5cm}
\lambda_{d_1\cdots d_n}^-
\prod_{i=1}^n
\widehat{f_i}(d_i)
\le
\prod_{i=1}^n
f_i(k_i)
\le
\hspace{-0,5cm}
\sum_{\substack{ 
d_i|k_i  \\ 
p|d_i \Rightarrow p \in \c{P}_i \\ 
d_i \ \text{squarefree}}}
\hspace{-0,5cm}
\lambda_{d_1\cdots d_n}^+
\prod_{i=1}^n
\widehat{f_i}(d_i)
.\]
\end{proposition}
\begin{proof} 
Let $n_i$ be squarefree integers composed entirely of primes in $\c{P}_i$
for all $i=1,\ldots,n$
and assume that they are coprime in pairs.
Then by $(1)$ and $(2)$ we get
\[1
=
\prod_{i=1}^n
\frac{\widehat{f_i}(1)}{f_i(1)}
=
\sum_{\substack{\b{m} \in \N^n\\ m_i|n_i }}
\l(\prod_{i=1}^n\frac{\widehat{f_i}(m_i)}{f_i(m_i)}\r)
\delta(m_1\cdots m_n),\]
where
$\delta$ is the characteristic function of $\{1\}$.
Noting that 
$\frac{\widehat{f_i}(m_i)}{f_i(m_i)}>0$
due to $(2)$ and $(4)$,
we deduce that 
\[
\sum_{\substack{\b{m} \in \N^n\\ m_i|n_i }}
\l(\prod_{i=1}^n\frac{\widehat{f_i}(m_i)}{f_i(m_i)}\r)
\l(\sum_{d|m_1\cdots m_n}
\lambda_{d}^-
\r)
\le 1 \le
\sum_{\substack{\b{m} \in \N^n\\ m_i|n_i }}
\l(\prod_{i=1}^n\frac{\widehat{f_i}(m_i)}{f_i(m_i)}\r)
\l(\sum_{d|m_1\cdots m_n}
\lambda_{d}^+
\r)
.\]
Each $d$ in the sum can be written uniquely as $d_1\cdots d_n$, where 
$d_i|m_i|n_i$,
and therefore
writing
$m_i=d_id_i^*$
allows us to
transform
the sums over $\b{m}$
into
\[
\sum_{\substack{\b{d} \in \N^n\\ d_i|n_i \\ \gcd(d_i,d_j)=1}}
\lambda_{d_1\cdots d_n}^\pm
\sum_{\substack{\b{m} \in \N^n\\ d_i|m_i \\ m_i|n_i }}
\l(\prod_{i=1}^n\frac{\widehat{f_i}(m_i)}{f_i(m_i)}\r)
=
\sum_{\substack{\b{d} \in \N^n\\ d_i|n_i  \\ \gcd(d_i,d_j)=1}}
\lambda_{d_1\cdots d_n}^\pm
\prod_{i=1}^n
\l(
\sum_{\substack{d_i^*\in \N\\ d_i^* |\frac{n_i}{d_i} }}
\frac{\widehat{f_i}(d_id_i^*)}{f_i(d_id_i^*)}
\r)
.\] 
Note that 
$m_i|n_i$ implies that  
$m_i$ is squarefree and hence
$\gcd(d_i,d_i^*)=1$.
By $(1)$
we get
\[
\sum_{\substack{d_i^*\in \N\\ d_i^* |\frac{n_i}{d_i} }}
\frac{\widehat{f_i}(d_id_i^*)}{f_i(d_id_i^*)}
=
\frac{\widehat{f_i}(d_i)}{f_i(d_i)}
\sum_{\substack{d_i^*\in \N\\ d_i^* |\frac{n_i}{d_i} }}
\frac{\widehat{f_i}(d_i^*)}{f_i(d_i^*)},
\]
and furthermore
an easy computation reveals that the sum over $d_i^*$ equals
$f_i(\frac{n_i}{d_i})^{-1}$ owing to the definition of $\widehat{f_i}$.
One also has that 
$\gcd(d_i,\frac{n_i}{d_i})=1$
since $n_i$ is squarefree, and hence
$(1)$ shows that 
$
f_i(n_i)=
f_i(d_i)
f_i\!\l(n_i/d_i\r).
$
We obtain that 
\[\frac{\widehat{f_i}(d_i)}{f_i(d_i)}
\sum_{\substack{d_i^*\in \N\\ d_i^* |\frac{n_i}{d_i} }}
\frac{\widehat{f_i}(d_i^*)}{f_i(d_i^*)}
=
\frac{\widehat{f_i}(d_i)}{f_i(n_i)} 
,\]
which proves the claim of our proposition in the case that
each $n_i$ is squarefree and composed only by primes in $\c{P}_i$.

We shall need
$(3)$ to prove our proposition in its full
generality. Given any positive integers
$k_1,\ldots, k_n$
as in the statement of our proposition,
define 
\[n_i:=\prod_{\substack{p | k_i \\ p \in \c{P}_i}}p.\]
Observe that 
by 
$(1),(3)$
and
$(5)$
we have
$
f_i(k_i)=
f_i(n_i)
$
and hence 
\[
\prod_{i=1}^n
f_i(k_i)=\prod_{i=1}^n
f_i(n_i)
\lesseqgtr
\sum_{d_i|n_i } 
\lambda_{d_1\cdots d_n}^\pm
\prod_{i=1}^n
\widehat{f_i}(d_i)
=
\hspace{-0,5cm}
\sum_{ \substack{
d_i|k_i  \\ 
p|d_i \Rightarrow p \in \c{P}_i \\ 
d_i \ \text{squarefree}}}
\hspace{-0,5cm}
\lambda_{d_1\cdots d_n}^\pm
\prod_{i=1}^n
\widehat{f_i}(d_i),
\]
which concludes our proof.
\end{proof}

\begin{lemma}
\lab{16+}
Assume that 
$n\in \N$,
$z\in \R_{>D_0}$
and let $P(z)=\prod_{D_0<p\leq z}p.$
Whenever 
$\b{k} \in (\Z-\{0\})^n$ with $\gcd(k_i,k_j)=1$ for all $i\neq j$,
we have 
\[
\prod_{i=1}^n
2^{-\omega(k_i;z)}
\ge
\sum_{\substack{ 
d_i|(k_i,P(z)) \\
\gcd(d_i,d_j)=1,  i\neq j }} 
\frac{\lambda_{d_1\cdots d_n}^-}
{\tau(d_1)\cdots\tau(d_n)}
.\]
\end{lemma}
\begin{proof}
Let $\c{P}_i$ be the set of primes in the interval
$(D_0,z]$ and
define the multiplicative
function $f_i(n)$ through
\[f_i(p^m)=
 \begin{cases} 
\frac{1}{2} &\mbox{if }  p|P(z),\\ 
1 &\mbox{otherwise,}
\end{cases}
\]
for all primes $p$ and $m\in \N$.
The conditions of Proposition~\ref{12++} are  
met and 
it is straightforward to check that
\[\widehat{f_i}(d)=
 \begin{cases} 
2^{-\omega(d)} &\mbox{if }  d|P(z),\\ 
0 &\mbox{otherwise,}
\end{cases}
\]
something which finalises our proof.     
\end{proof}
Define for each $i=1,\ldots,n$,
\[
\c{A}_i:=\langle \Delta_i \rangle
\l(1+\deg(\Delta_i)\r)
\|\c{R}\|_\infty^{\deg(\Delta_i)}
\]
and observe that for each $(s,t) \in \c{B}$
and
$i=1,\ldots,n$,
we have
$|\Delta_i(s,t)|\leq \c{A}_i B^{\deg(\Delta_i)}$.
Injecting  
Lemma~\ref{16+} into~\eqref{15}
produces
 the following result.
\begin{lemma}
\lab{19.a}
Assume that  
${\varpi}>\frac{\log D_0}{\log B}100 n(n+1)$
and let us define
\[
\c{D}:=
\left\{\b{d}\in \N^n:
\begin{array}{l}
d_i \leq \c{A}_i B^{\deg(\Delta_i)}, \mu(d_i)^2=1,
p|d_i\Rightarrow  p\leq B^{\frac{{\varpi}}{100n(n+1)}}, \\
\gcd(d_i,W)=1, \gcd(d_i,d_j)=1, i\neq j
\end{array}
\right\}
\]
and for $\b{d} \in \c{D}$ let
\[
M_{\b{d}}^*(B)
:=
\sum_{\substack{(s,t)\in \c{B} \\ d_i|\Delta_i(s,t) }}
\prod_{i=1}^n
r_i^*(s,t)
.\]
Then we have
\[
N(\pi,B)
\gg 
\sum_{
\b{d} \in \c{D}}
\frac{\lambda_{d_1\cdots d_n}^-}
{\tau(d_1) \cdots \tau(d_n)}
M_{\b{d}}^*(B)
,\]
where the implied constant is allowed to depend on $\varpi$.
\end{lemma}

\subsection{Transition to $r$ functions}
\lab{19.b}
Our aim in this section is to
replace the $r^*$ functions by functions that are amenable to divisor sum techniques.
To this aim let us define for all $(s,t) \in \c{B}$,
$i=1,\ldots,n$, and $m \in \N$ with
$m|\Delta_i(s,t)$ the function
\[r_i(s,t;m):=\sum_{\substack{k|\Delta_i(s,t)/m\\ \gcd(k,W)=1}}
\l(\frac{F_i(s,t)}{k}\r)
.\]
For any
$a\in \Z-\{0\}$, the Jacobi symbol $(\frac{a}{k})$ is multiplicative
with respect to $k$, and therefore,
whenever $A\in \Z-\{0\}$, we are provided with
\[\prod_{\substack{p|A\\ p\nmid W}}\l(1+\l(\frac{a}{p}\r)\r)=
\sum_{\substack{m^2|A \\ \gcd(m,W)=1}}
\mu(m)
\sum_{\substack{k|A/m^2\\ \gcd(k,W)=1}}
\l(\frac{a}{k}\r)
,\] which implies that  
for all $(s,t) \in \c{B}$
and
$i=1,\ldots,n$, 
one has 
\[r_i^*(s,t)=\sum_{\substack{m_i^2|\Delta_i(s,t) \\ \gcd(m_i,W)=1}}
\mu(m_i)
r_i(s,t;m_i^2)
.\]
Letting 
\[
\c{M}:=
\left\{\b{m}\in \N^n:
\begin{array}{l} 
m_i \leq \c{A}_i^{1/2}
B^{\deg(\Delta_i)/2}, 
\mu(m_i)^2=\gcd(m_i,W)=1,\\
\gcd(m_i,m_j)=\gcd(m_i,d_j)=1, i\neq j 
\end{array}
\right\}
\] and defining
for $(\b{d},\b{m}) \in \c{D}\times\c{M}$,
\beq{s d sum}{
S_{\b{d}}(B;\b{m})
:=
\sum_{\substack{(s,t) \in \c{B}\\ [m_i^2,d_i]|\Delta_i(s,t)}}
\prod_{i=1}^n
r_i(s,t;m_i^2) 
,} 
allows us to infer the validity of
\beq{teaa}{
M_{\b{d}}^*(B)=
\sum_{\b{m} \in \c{M}}
\mu(m_1\cdots m_n) 
S_{\b{d}}(B;\b{m})
.}
Let us define for all
$Y \in \R_\geq 1$
and $\b{d} \in \c{D}$,
\[
M_{\b{d}}(B,Y)
:=
\sum_{\substack{\b{m} \in \c{M}\\\|\b{m}\|\leq Y}}
\mu(m_1\cdots m_n) 
S_{\b{d}}(B;\b{m})
.\]
\begin{lemma}
\lab{cut-off 2}
We have for all
$Y \in \R \cap (1,B^{1/2})$,
$\b{d} \in \c{D}$
and any
$\varepsilon>0$,  
\[
M_{\b{d}}^*(B)=
M_{\b{d}}(B,Y)
+O_\varepsilon\!\l( 
\frac{B^{2+\varepsilon}}{Y}
\r)
.\]
\end{lemma}
\begin{proof}
Each $m_i$ in~\eqref{teaa}
satisfies
$m_i^2\leq |\Delta_i(s,t)| \le \c{A}_i B^{\deg(\Delta_i)}$
and hence 
the familiar estimate
$\tau(n)\ll_\varepsilon B^\varepsilon$
shows that
\beq
{heroes_bowie}
{
M_{\b{d}}^*(B)-
M_{\b{d}}(B,Y)
\ll
B^\varepsilon
\sum_{i=1}^n \ 
\sum_{\substack{\b{m} \in \c{M} \\ m_i >Y }}
L(\b{m}),}
where
\[L(\b{m}):=
\#\Big\{(s,t) \in \Zp^2\cap B\c{R}: m_i^2 |\Delta_i(s,t)\Big\}
.\]
We next show that $L(\b{m})$ is contained in a union of at most 
$\ll_\varepsilon B^\varepsilon$ lattices in $\Z^2$
of determinant 
$m_1^2\cdots m_n^2$.
This is trivially the case if there exists a prime $p$ 
dividing one of the $m_i$
such that  
$\Delta_i(s,t)\equiv 0 \md{p^2}$ has no solution with $(s,t) \in \Zp^2$.
In the opposite case, for each such prime $p$
there exist at most
\[\#\{\xi\md{p^2}:\Delta_i(\xi,1)\equiv 0\md{p^2} \ \text{or} \ 
\Delta_i(1,\xi)\equiv 0\md{p^2}
\}
\ll_\varepsilon B^\varepsilon
\]
values of $\xi$ such that all $(s,t)$ in the aforementioned set 
belong in the lattice determined by the condition 
$s\equiv \xi t\md{p^2}$ or
$t\equiv \xi s \md{p^2}$
and whose determinant is $p^2$.
Noting that 
$\b{m} \in \c{M}$ and hence the integers $m_i$ are coprime in pairs,
we can combine these lattices in a single lattice $\Lambda$ of determinant
$m_1^2\cdots m_n^2$.
The
estimate
\[
\#\Big\{\Z^2 \cap \Lambda \cap B\c{R}\Big\}\ll
\frac{B^2 \vol(\c{R})}{|\det(\Lambda)|}
+1\]
reveals that 
\[L(\b{m})\ll_\varepsilon
B^\varepsilon
\l(\frac{B^2}{m_1^2\cdots m_n^2}+1\r),
\]
which, once injected into~\eqref{heroes_bowie}, provides the proof of our lemma
by virtue of the fact that
\[m_i \leq \c{A}_i^{1/2}B^{\deg(\Delta_i)/2}\] for each $\b{m} \in \c{M}$.
\end{proof}
Our aim in~\S\ref{stephan d}
is to evaluate the sums
$S_{\b{d}}(B;\b{m})$.
We choose
to state the end result of~\S\ref{stephan d}
here and explore its 
implications regarding
$M_{\b{d}}^*(B)$ immediately.

Before stating the result some notation is required.
Let us introduce the 
following set
of multiplicative 
functions
\[\c{U}:=\Big\{
f:\N\to \mathbb{R}_{>0}:
f(n)=\prod_{p|n}(1+g(p))
\
\text{where} \ g(p)\ll_{f} p^{-1}
\Big\}
.\]
Notice that $\c{U}$ is a group under pointwise multiplication
with identity given by the constant function $f(n)=1$.
Furthermore, for each element $f \in \c{U}$ there exists
$a>0$, such that  
for all $\varepsilon>0$ we have
\beq{pointw}{f(n)\ll \tau(n)^a\ll_\varepsilon n^\varepsilon
.}We shall
find it convenient to define for each $m \in \N$ and $f \in \c{U}$ the function
\beq{conv}{
f(n,m)=\sum_{\substack{d|n\\ \gcd(d,m)=1}}\mu(d)^2g(d)
,} notice that for all $(f(n),m) \in \c{U}\times \N$ we have $f(n,m) \in \c{U}$.
It is straightforward to verify
\beq{genesis}
{f\l(\prod_{i=1}^r n_i\r)=f(n_1)
\prod_{i=2}^r 
f\!\l(n_i,\prod_{j=1}^{i-1} n_j\r)
}
for all $\b{n} \in \N^r$ and $f \in \c{U}$,
while the property
\beq{all+}{
f(\gcd(n,m))=\frac{f(n)f(m)}{f(nm)}
\
\
(n,m \in \N,f\in \c{U}),
} is obviously valid.

Let us define
for all $d \in \N$
coprime to $W$ and each
$i=1,\ldots,n$,
\[ 
\tau_i(d):=
\hspace{-0,5 cm}
\sum_{\substack{\xi \md{d}\\\Delta_i(\xi,1)\equiv 0
\md{d}}}
\hspace{-0,5 cm}
1,
\hspace{1 cm}
\varrho_i(d):=
\hspace{-0,5 cm}
\sum_{\substack{\xi \md{d}\\
\Delta_i(\xi,1)\equiv 0
\md{d}}}
\hspace{-0,5 cm}
\left(\frac{F_i(\xi,1)}{d}\right).
\]
Extending both functions to $\N$ by letting them vanish when the argument has a prime divisor $p\leq D_0$,
allows one to show that both $\tau_i$ and $\varrho_i$ are multiplicative.
Assuming that $D_0$ is large enough,
Hensel's lemma can be utilised to prove
that for all primes $p>D_0$ and $k\in \N$
we have 
\[|\varrho_i(p^k)|\leq \tau_i(p^k)=\tau_i(p)\leq \deg(\Delta_i),\]
and therefore
there exists $A_i>1$ such that for all 
$\varepsilon>0$
and
$d \in \N$,
 the estimates
\beq{national acrobat}
{
\tau_i(d), \varrho_i(d) \ll 
\tau(d)^{A_i}
\ll_\varepsilon
d^\varepsilon
} hold.
Defining 
the arithmetic
functions
\beq{bowie_low}{
\sigma_i(d):=\prod_{p|d}\l(\tau_i(p)+\sum_{k=0}^\infty \varrho_i(p^{k+1})p^{-k}\r)
\
\
\text{and}
\
\
\tau^\sharp_i(d):=\prod_{p|d}\l(\tau_i(p)+\sum_{k=1}^\infty \varrho_i(p^k)p^{-k}\r),
}
enables us to state the main result of~\S\ref{stephan d}.
\begin{theorem}
\label{b evans}
There exist
functions $g_i,h_i \in \c{U}$
and
a positive constant $c$,
such that
for each
$(\b{d},\b{m}) \in \c{D} \times \c{M}$,
we have
\[S_{\b{d}}(B;\b{m})=c B^2
\l(\prod_{i=1}^n
\frac{\sigma_i(d_i)}{g_i(d_i)d_i}\r)
\l(\prod_{i=1}^n 
\frac
{\tau^\sharp_i(m_i/\gcd(d_i,m_i))}
{h_i(m_i,d_i)m_i^2\gcd(d_i,m_i)^{-1}}
\r)
+O\!\l(\|\b{d}\|^n
B^{\frac{19}{20}}\r),\]
where the implied constant is independent of $B,\b{d}$ and  $\b{m}$.
\end{theorem} 
Using Theorem~\ref{b evans} we can now deduce the following result.
\begin{proposition}
\lab{r2d} 
There exist
$u_i \in \c{U},c'>0$,
such that  
one has
for all $\b{d} \in \c{D}$, 
\[M_{\b{d}}^*(B)
=c' B^2
\l(\prod_{i=1}^n\frac{\sigma_i(d_i)}{d_i}u_i(d_i)
\r)
+O\!\l(
\|\b{d}\|^{n}
B^{2-\frac{1}{50n}}
\r)
,\]
where the implied constant is independent of $B$ and $\b{d}$.
\end{proposition} 
\begin{proof} 
Letting
\[A(\b{d})=
\sum_{\substack{\b{m} \in \c{M}\\
\|\b{m}\|\leq Y}}
\mu(m_1\cdots m_n) 
\prod_{i=1}^n 
\frac
{\tau^\sharp_i(m_i/\gcd(d_i,m_i))}
{h_i(m_i,d_i)m_i^2}
\gcd(d_i,m_i),
\]
and using Lemma~\ref{cut-off 2} with $Y=B^{\frac{1}{20(1+n)}}$,
allows us to infer
upon employing Proposition~\ref{b evans},
that
$M_{\b{d}}^*(B)$
equals
\beq{all}{
c B^2
\l(\prod_{i=1}^n
\frac{\sigma_i(d_i)}{g_i(d_i)d_i}\r)
A(\b{d}),
}
up to an admissible error term.
Define the constants
\[\eta(p)=\sum_{i=1}^n\frac{\tau^\sharp_i(p)}{h_i(p)}
\
\
\text{and}
\
\
\eta_j(p)=p+\sum_{i\neq j}\frac{\tau^\sharp_i(p)}{h_i(p)}
.\]
Using~\eqref{national acrobat} provides the estimate
\[\sum_{m>Y}\frac
{\tau^\sharp_i(m/\gcd(d,m))}
{h_i(m,d)m^2}
\gcd(d,m)\ll_\varepsilon
\sum_{m>Y} \frac{m^\varepsilon}{m^2}\gcd(d,m)
\ll_\varepsilon d^\varepsilon Y^{\varepsilon-1},
\]
which
shows that the condition $\|\b{m}\|\leq Y$ in $A(\b{d})$ can be removed harmlessly.
We therefore end up with a series whose Euler product
is
\[
\prod_{p\nmid W}\l(1-\frac{\eta(p)}{p^2}\r)
\prod_{i=1}^n
\prod_{p|d_i}\frac{1-\frac{\eta_i(p)}{p^2}}{1-\frac{\eta(p)}{p^2}}
=c''\prod_{i=1}^nh_i'(d_i), \ \text{say}.\]
The estimates $\eta(p)\ll 1$ and $\eta_i(p)\ll p$
imply that, once $W$ is enlarged, we can safely assume $c''>0$ and $h_i' \in \c{U}$.
Letting $c'=cc''$
and
$u_i=h_i'/g_i$ completes our proof. 
\end{proof}
\section{Evaluation of the ensuing divisor sums} 
\lab{stephan d}
Our aim
in~\S\ref{stephan d}
is to prove Theorem~\ref{b evans}.
\subsection{Auxiliary results}
We may use
Hensel's lemma to deduce that for all 
$ \nu \in \N$
and
primes $p>D_0$, 
\beq{bi1231}{ 
\varrho_i(p^\nu)=
\begin{cases} 
\varrho_i(p) &\mbox{if} \  \nu \ \mbox{is odd},\\ 
\tau_i(p) &\mbox{otherwise,}
\end{cases}
}
from which we infer that the function
\beq{immodop0}{
R_i(p,z):=
\sum_{\nu=1}^\infty
\varrho_i(p^\nu)z^\nu,
}
defined for
$z\in \mathbb{C}$ with
$|z|\le 2^{-1/2}$ and primes $p$,
satisfies
\beq{bwv 244}
{R_i(p,z)=
z\frac{\varrho_i(p)+z\tau_i(p)}{1-z^2}
.}
We can similarly prove the
identity
\beq{bwv 988}
{\sigma_i(p)=
\tau_i(p)+\varrho_i(p)
+\frac{p\tau_i(p)+\varrho_i(p)}{p^2-1}
.}
For a fixed $i=1,\ldots,n$ we
let $c=\Delta_i(1,0)$ and we immediately see that 
$\theta:=\theta_i/\Delta_i(1,0)$
is a root of the irreducible integer polynomial
$\Delta_i(x,c)/c$, and hence an algebraic integer.
Therefore, 
letting
\[
k=\Q(\theta)
\
\text{and}
\
K=\Q(\theta,\sqrt{F_i(\theta,c)}),
\]
we see that
$\Q(\theta_i)=\Q(\theta)$,
thus leading to
$[K:k]=2$
via~\eqref{splitt}.
The non-trivial 
irreducible
representation
of $\mathrm{Gal}(K/k)$ 
gives rise to the entire Artin $L$-function
\[
L_i(s)=\prod_{\mathfrak{p}\trianglelefteq
\mathcal{O}_k
}
\left(
1-\chi(\mathfrak{p})N_{k}(\mathfrak{p})^{-s}
\right)^{-1},
\]
that does not vanish at $s=1$
and
where
$\chi(\mathfrak{p})$
is $0$ if $\mathfrak{p}$ is ramified in $K$ and 
$1$ or $-1$ according to if $\mathfrak{p}$
is split or inert in $\mathcal{O}_K$. 
Enlarging $W$ allows us to use 
the well-known principle of Dedekind,
according to which,
linear factors $x-\xi \in
\F_p[x]$
of $\Delta_i(x,c)$ 
parametrise
the linear prime ideals $\p_\xi $ 
above $p$.
The map 
$\Z[\theta]/(p,\theta-\xi)\to\c{O}_k/\p_\xi \c{O}_k$,
given by
\[
h(\theta)+(p,\theta-\xi)
\mapsto
h(\theta)+\p_\xi \c{O}_k,
\]
is an isomorphism,
and
therefore the image of
$F_i(\theta,c)$
is
a square in
$\c{O}_k/\p_\xi
\c{O}_k
\simeq 
\F_p$
if and only if the prime
$\p_\xi$ splits in $K$.
Noting that $\deg(F_i)$ is even,
we are led to
\beq{golden years}
{\sum_{\substack{\mathfrak{p}|(p)
\\N_k(\mathfrak{p})=p}}
\chi(\mathfrak{p})
=
\sum_{
\substack{
\xi \md{p} \\ \Delta_i(\xi,c)\equiv 0 \md{p}}}
\left(
\frac{F_i(\xi,c)}{p}
\right)=\varrho_i(p).
}
Letting 
\beq{alladin sane}
{P_i(s):=\prod_{p \nmid W}
\left(1+R_i(p,p^{-s})
\right),}
we see
that the analytic properties
of the
function
$G_i= P_i L_i^{-1}$
are determined solely by
prime ideals with $N_{k}(\mathfrak{p})<p^2$.
In particular this
enables us 
to show that $G_i$
is holomorphic 
in the region
$\Re(s)>\frac{1}{2}$ and
that for any $\varepsilon>0$ it satisfies 
$G_i(s)\asymp_\varepsilon 1$ 
whenever
$\Re(s)>\frac{1}{2}+\varepsilon$. 
We may thus deduce that 
\beq{ippo}{P_i(1)\neq 0}
and that for all $\varepsilon>0$,
\beq{krateio}{
P_i(s)\ll_\varepsilon
|\Im(s)|^{\frac{\deg(\Delta_i)}{2}(1-\Re(s))+\varepsilon},
\
\text{whenever}
\
\frac{1}{2}<\Re(s)\leq 1,}
the last property being due to the convexity bound~\cite[Eq.(5.20)]{iwa} applied to $L_i(s)$.

Observe that~\eqref{golden years}
and the prime number theorem for the
$L$-function $L_i(s)$  
reveal that
\beq{mertenoulis}{\sum_{p<z}\frac{\varrho_i(p)}{p}=
\sum_{N_{k}(\mathfrak{p})< z}\frac{\chi(\mathfrak{p})}{N_{k}(\mathfrak{p})}+
b_i+O\l(\frac{1}{\log z}\r)
=b'_i+O\l(\frac{1}{\log z}\r),}
for some constants $b_i$ and $b'_i$, both independent of $z$.
We similarly have
\[
\tau_i(p)=
\sum_{\substack{\mathfrak{p}|(p)
\\N_k(\mathfrak{p})=p}}1
\]
and therefore the prime number theorem for the Dedekind zeta function
of $k$ leads to the estimate
\beq{mertenoulis'}{\sum_{p<z}\frac{\tau_i(p)}{p}=\log \log z+b''_i+O\l(\frac{1}{\log z}\r),}
for some constant $b''_i$ independent of $z$.

\subsection{Preparations} 
\label{lejeune}
Recall the definition of the
divisor sums
$S_\b{d}(B;\b{m})$,
introduced in~\eqref{s d sum}.
The fact
$\Delta_i(s_0,t_0)\neq 0$
guarantees that
for each prime $p$
the sequence 
$\nu_p(\Delta_i(s,t))$
stabilises when 
$(s,t)\in \c{B}$ and
$(s,t)\equiv (s_0,t_0)\md{p^n}$ for $n\to \infty$.
Therefore enlarging $W$ allows us to
define integers $W_i$, independent of $s$ and $t$,
composed of primes $p\leq D_0$
such that if
$(s,t)\equiv (s_0,t_0)\md{W}$  
then $\Delta_i(s,t)^\dagger=|\Delta_i(s,t)|/W_i$.
Defining 
for all 
$(\b{d},\b{m}) \in \c{D} \times \c{M}$,
$\b{v} \in \l(\R_{\geq 1}\r)^n$
and ordered
$\bb{\psi} \in \{0,1\}^n$,
the entities 
\beq{def:qi}{Q_i:=\l(\frac{\c{A}_i B^{\deg(\Delta_i)}}{m_i^2W_i}\r)^{1/2},}
\[\c{K}:=
\left\{\b{k}\in \N^n:
\begin{array}{l} 
1\leq k_i \leq Q_i, \gcd(k_i,W)=1,\\ 
\gcd(k_i,d_jm_jk_j)=1, i\neq j
\end{array}
\right\},
\]
\[
\c{R}_{\bb{\psi}}\!\l(\b{v}\r)
=\bigcap_{i=1}^n
\Big\{\left(s,t\right) \in \R^2
\cap
B\mathcal{R}:
|\Delta_i(s,t)|\geq \psi_i v_i m_i^2 Q_iW_i
\Big\},
\]
and
\beq{def:omom}{\omega_{\bb{\psi}}\!\l(\b{v}\r)=\vol(\c{R}_{\bb{\psi}}\!\l(\b{v}\r)),}
we have the following result.
\begin{lemma}
[Hyperbola trick]
\label{hyper0}
For each
$(\b{d},\b{m}) \in \c{D} \times \c{M}$
we have
\[S_\b{d}(B;\b{m})=
\sum_{\bb{\psi} \in \{0,1\}^n} 
\sum_{\substack{\b{k} \in \c{K} \\
(1-\psi_i)k_i\neq Q_i\\
}} 
\sum_{\substack{
(s,t) \in \c{B}\cap \c{R}_{\bb{\psi}}(\b{k})\\
q_i|\Delta_i(s,t)}}
\
\prod_{i=1}^n
\l(\frac{F_i(s,t)}{k_i}\r)
,\]
where
$q_i=[d_i,k_im_i^2].$
\end{lemma}
\begin{proof}
The $p$-adic stability property alluded to earlier, reveals that 
for all $(s,t) \in \c{B}$ one has
$r_i(s,t;m_i^2)=r_i(s,t;m_i^2W_i)$,
and hence letting 
\[
r^{(0)}_i\!\l(s,t\r)
=\sum_{\substack{k_i| \frac{\Delta_i(s,t)}{m_i^2W_i}\\k_i<Q_i}}
\l(\frac{F_i(s,t)}{k_i}\r)
\
\
\text{and}
\
\
r^{(1)}_i\!\l(s,t\r)
=\sum_{\substack{k_i^*| \frac{\Delta_i(s,t)}{m_i^2W_i}\\k_i^*\leq \frac{|\Delta_i(s,t)|}{Q_im_i^2W_i}}}
\l(\frac{F_i(s,t)}{k_i^*}\r)
,\]
we are driven via
Proposition~\ref{let it be}
towards
\[r_i(s,t;m_i^2)=
\sum_{k_ik_i^*=\frac{|\Delta_i(s,t)|}{m_i^2W_i}}
\l(\frac{F_i(s,t)}{k_i}\r)=
r^{(0)}_i\!\l(s,t\r)
+r^{(1)}_i\!\l(s,t\r)
.\]
The proof of our lemma is furnished upon
injecting this equality into~\eqref{s d sum}
and noticing that 
$[[d_i,m_i^2],k_im_i^2]=[d_i,k_im_i^2]$
due to the fact that both $d_i$ and $m_i$ are squarefree integers.
\end{proof}  
Let us
define
for
$\bb{\psi} \in \{0,1\}^n,
\b{k}\in \c{K}$ and 
$\boldsymbol{\xi}\in \prod_{i=1}^n(\Z/q_i\Z)$
the set
\[
\c{L}_{\bb{\psi}}\!\l(\b{k},\boldsymbol{\xi}\r)
:=\#\left\{(s,t)\in \Zp^2\cap \c{R}_{\bb{\psi}}\!\l(\b{k}\r):
\begin{array}{l}
(s,t)\equiv (s_0,t_0) \md{W},\\
s\equiv \xi_i t \md{q_i}
\end{array}
\hspace{-0,1cm}
\right\}
.\]
\begin{lemma}
[Partitioning in lattices]
\lab{miles}
For each
$(\b{d},\b{m}) \in \c{D} \times \c{M}$
we have
\[S_\b{d}(B;\b{m})=
\sum_{\bb{\psi} \in \{0,1\}^n} 
\sum_{\substack{\b{k} \in \c{K} \\
(1-\psi_i)k_i\neq Q_i\\
}}  
\sum_{\substack{\xi_i\md{q_i}\\
\Delta_i(\xi_i,1)\equiv 0 \md{q_i} 
}}
\prod_{i=1}^n
\l(\frac{F_i(\xi_i,1)}{k_i}\r)
\c{L}_{\bb{\psi}}\!\l(\b{k},\boldsymbol{\xi}\r)
.\]
\end{lemma}
\begin{proof}

As in the beginning of the proof of 
Lemma~\ref{soll}, one sees that
$\gcd(q_i,W)=1$ and
$q_i|\Delta_i(s,t)$ 
imply that
$s/t\md{q_i}$
is well-defined.
We deduce that for all $(s,t)$ appearing in the statement of Lemma~\ref{hyper0}
we have
$s\equiv \xi_i t\md{q_i}$ for all $i=1,\ldots, n$,
where $\xi_i$ is an integer satisfying the condition
$\Delta_i(\xi_i,1)\equiv 0
\md{q_i}$.
The property
$k_i|q_i$ 
validates the equality
\[
\l(\frac{F_i(s,t)}{k_i}\r)=
\l(\frac{F_i(\xi_i t,t)}{k_i}\r)
=
\l(\frac{t^{\deg(F_i)}
F_i(\xi_i,1)}{k_i}\r)
,\]
and 
the fact that 
$2|\deg(F_i)$ shows that 
\[\l(\frac{t^{\deg(F_i)}
F_i(\xi_i,1)}{k_i}\r)
=
\l(\frac{
F_i(\xi_i,1)}{k_i}\r),\]
an observation which concludes our proof.
\end{proof}
\subsection{Point counting in primitive lattices} 
We call a lattice $G\subset \Z^2$ 
\textit{primitive}
if the only integers fulfilling
$G\subset\delta \Z^2$
are $\delta=\pm 1$.
\begin{lemma}
\lab{lattice 2}
Let 
$\b{A}\in \Z^{2 \times 2}$
be an upper triangular matrix
of non-zero determinant
and 
consider the lattice given by
$G=\{\b{A}\b{y}:\b{y}\in \Z^2\}$.
We shall let its  
determinant 
and
first successive minimum
be denoted by 
$\det(G)$
and
$\lambda_1(G)$ respectively.
Assume that $\c{V}\subset \R^2$ is a bounded measurable set whose boundary is piecewise differentiable
and
that $\b{x}_0 \in \Z^2$ and 
$q \in \Z$
are such that 
$\gcd(\b{x}_0,q)=1$
and 
$\gcd\l(\det(G),q\r)=1$.
\vskip 0.2cm 
$(1)$
The estimate
\[\#\Big\{\b{x}\in \c{V}\cap G:\b{x}\equiv \b{x}_0 \md{q}\Big\}=
\frac{\vol(\c{V})}{\det(G)}\frac{1}{q^2}
+O\!\l(1+\frac{\partial(\c{V})}{\lambda_1(G) q}\r)
\] holds
with an absolute
implied constant.

$(2)$
If  $G$ is primitive then the quantity
$\#\Big\{\b{x}\in \Zp^2 \cap \c{V}\cap G:\b{x}\equiv \b{x}_0 \md{q}\Big\}$
equals \[ 
\frac{\vol(\c{V})}{\zeta(2)\det(G)q^2}
\prod_{p | \det(G)} 
\l(1+\frac{1}{p}\r)^{-1}
\prod_{p | q} 
\l(1-\frac{1}{p^2}\r)^{-1}
\]
up to an error
\[
\ll
\frac{\tau(\det(G))}{\lambda_1(G)}
\l(
\|\c{V}\|_\infty 
+
\frac{\partial(\c{V})}{q}\log\l(1+\|\c{V}\|_\infty\r)
+
\frac{\vol(\c{V})}{q^2}
\frac{1}{1+\|\c{V}\|_\infty}
\r),
\]
with an absolute implied constant.
\end{lemma} 
\begin{proof}
$(1)$
Defining $\b{y}_0$
through
$\det(G)\b{y}_0\equiv \b{A}^{\mathrm{adj}}\b{x}_0\md{q}$
and letting $\b{y}=\b{y}_0+q\b{z}$
we see that the quantity
under consideration is
$\#\{\b{z}\in \Z^2: \b{A}\b{z} \in (\c{V}-\b{A}\b{y}_0)/q\}$.
Our proof is then concluded by
observing that 
the region $(\c{V}-\b{A}\b{y}_0)/q$
has volume $\frac{\vol(\c{V})}{q^2}$
and length of boundary $\frac{\partial(\c{V})}{q}$.
\newline
$(2)$ We shall deploy M\"{o}bius inversion to detect the coprimality condition
and
due to
$\gcd(\b{x},q)=1$ and
$\b{x}\equiv \b{x}_0 \md{q}$,
we only need to ensure that
$\b{x}$ is coprime to integers coprime to $q$.
Sieving out multiples of $\det(G)q$, 
we see that the quantity under consideration equals
\[
\sum_{\substack{m \in \N\\ \gcd(m,\det(G)q)=1}} \hspace{-0,5cm} \mu(m)
\#\Big\{\b{y}\in \Z^2:m|\b{M}\b{y}, \gcd(\b{A}\b{y},\det(G))=1, \b{y}\equiv \b{y}_0 \md{q}, 
\b{A}\b{y} \in \c{V}\Big\}
.\]
Due to $\gcd(m,\det(G))=1$,
the condition
$m|\b{A}\b{y}$
is equivalent to
$m|\b{y}$
and, letting $\b{y}=m\b{z}$, we obtain
\[
\sum_{\gcd(m,\det(G)q)=1} \hspace{-0,5cm} \mu(m)
\#\Big\{\b{z}\in \Z^2:\gcd(\b{A}\b{z},\det(G))=1, \b{z}\equiv \overline{m} \b{y}_0 \md{q}, 
\b{A}\b{z} \in \c{V}/m\Big\}
.\]
Notice that in order for the inner quantity to be non-zero 
we must have
$\lambda_1(G) \leq  \|\c{V}\|_\infty/m.$
Removing the condition $\gcd(\b{A}\b{z},\det\!\l(G\r))=1$
reveals that the quantity under consideration equals
\[
\sum_{k|\det(G)}\mu(k)
\sum_{\substack{m \leq  \|\c{V}\|_\infty /\lambda_1(G) \\ \gcd(m,\det(G)q)=1}} \hspace{-0,5cm} \mu(m)
\#\Big\{\b{z}\in \Z^2: 
k|\b{A}\b{z},
\b{z}\equiv \overline{m}  \b{y}_0 \md{q}, \b{A}\b{z} \in \frac{\c{V}}{m}\Big\}.
\]
The set
$G(k):=
\Big\{\b{A}\b{z}:\b{z}\in \Z^2, 
k|\b{A}\b{z}\Big\}
$
is a sublattice of
 $G=\Big\{\b{A}\b{z}:\b{z}\in \Z^2\Big\}$,
whose first successive minimum
$\lambda_1(G(k))$ satisfies
$k\lambda(G)\leq \lambda(G(k))$,
since if $\b{v}$
is a minimal non-zero integer vector of
$G(k)$ then $k|\b{v}$ and furthermore $\b{v}/k \in G$.
Assume that $\b{A}$ is given by
\begin{equation}
\lab{herm}
\b{A}=\left( \begin{array}{cc}
a_1 & a_2  \\
0 & a_3 \end{array} \right).
\end{equation}
We shall prove 
$\det\!\l(G(k)\r)=k\det\!\l(G\r)$
by making use of~\eqref{herm}.
Indeed, factorising the squarefree integer $k$,
known to divide $\det\!\l(G\r)$,
as
$k=k_1k_2$,
where $k_1$ has all of the prime factors of $k$
dividing
$a_1$, and using the primitivity of the lattice,
makes apparent 
that the property $k|\b{A}\b{z}$ is equivalent to
\[k_1|z_2
\
\
\&
\
\
z_1\equiv 
b
z_2 \md{k_2},
\]
where
$a_1b\equiv a_2\md{k_2}$.
Writing
$z_2=k_1u_1$ and
$z_1=b k_1u_1+k_2 u_2$
for some $\b{u} \in \Z^2$,
and bringing into play the matrix
\[
\b{A}'
:=\left( \begin{array}{cc}
(a_1b+a_2)k_1 & a_1k_2  \\
a_3k_1 & 0 \end{array} \right),
\]
gives birth to
$
\b{A}\b{z}=
\b{A}'
\b{u}
$. This equality allows us to
acquire
the desired result
\[
\det\l(G(k)\r)=\det\l(\b{A}'\r)=k
\det(G).
\]
We next have
\begin{align*}
&\#\Big\{\b{z}\in \Z^2: 
k|\b{A}\b{z},
\b{z}\equiv \overline{m}  \b{y}_0 \md{q}, \b{A}\b{z} \in \c{V}/m\Big\}
\\
=&\#\Big\{\b{u}\in \Z^2: 
\b{u}\equiv 
\b{A}
{\b{A}'}^{\mathrm{adj}}
\overline{a_1a_3km}  \b{y}_0 \md{q}, \b{A}'\b{u} \in \c{V}/m\Big\}
\end{align*}
and therefore
the first part of the present lemma confirms
that  
\begin{align*}
&\#\Big\{\b{x}\in \Zp^2:\b{x}\equiv \b{x}_0 \md{q}, \b{x} \in \c{V}\Big\}
\\
=&\sum_{k|\det(G)}\mu(k)
\sum_{\substack{m \leq  \|\c{V}\|_\infty /\lambda_1(G) \\ \gcd(m,\det(G)q)=1}} \hspace{-0,5cm} \mu(m)
\l(
\frac{\vol(\c{V})}{k\det\!\l(G\r)}\frac{1}{m^2q^2}
+O\!\l(1+\frac{\partial(\c{V})}{\lambda_1(G(k)) mq}\r)
\r)
.\end{align*}
Using the inequality 
$k\lambda_1(G)\leq \lambda_1(G(k))$
shows that the error term above is
\begin{align*}
&\ll
\sum_{k|\det(G)} 
\sum_{m \leq  \|\c{V}\|_\infty /\lambda_1(G)}
\l(1+\frac{\partial(\c{V})}{\lambda_1(G) qkm}\r)
\\
&=\tau(\det(G))\frac{\|\c{V}\|_\infty}{\lambda_1(G)}
+\tau(\det(G))
\frac{\partial(\c{V})}{\lambda_1(G) q}\log\l(1+\|\c{V}\|_\infty\r),
\end{align*}
while the main term equals
\begin{align*}
&\frac{\vol(\c{V})}{\det\!\l(G\r)q^2}
\sum_{k|\det(G)} \frac{\mu(k)}{k}
\hspace{-0,5cm}
\sum_{\substack{m \leq  \|\c{V}\|_\infty /\lambda_1(G) \\ \gcd(m,\det(G)q)=1}} \hspace{-0,5cm} 
\frac{\mu(m)}{m^2}=\\
&\frac{\vol(\c{V})}{\zeta(2)\det\!\l(G\r)q^2}
\sum_{k|\det(G)} \frac{\mu(k)}{k}
\prod_{p | \det(G)q} 
\l(1-\frac{1}{p^2}\r)^{-1}
+O\!\l(
\frac{\vol(\c{V})}{\det\!\l(G\r)}
\frac{\lambda_1(G)\tau(\det(\Lambda))}{q^2 \|\c{V}\|_\infty}
\r),
\end{align*}
hence,
utilising $\frac{1}{\det\l(G\r)}\ll \frac{1}{\lambda_1(G)^2}$
delivers the proof of our claim. 
\end{proof}
\subsection{Error term}
Recall the definition of 
$\omega_{\bb{\psi}}$
in~\eqref{def:omom}
and define the multiplicative function 
\[f_0(k)=\prod_{p|k}\l(1+\frac{1}{p}\r)^{-1}.\]
\begin{lemma}
\lab{laatt}
There exists 
$c_0>0$
such that  
\[
\c{L}_{\bb{\psi}}\!\l(\b{k},\boldsymbol{\xi}\r)
=
c_0 \
\omega_{\bb{\psi}}(\b{k})
\prod_{i=1}^n 
\frac{f_0(m_ik_id_i)\gcd(d_i,k_im_i)
}{m_i^2k_id_i} 
+O\!\l(\frac{\tau(
m_1^2d_1k_1
\ldots
m_n^2d_nk_n
)}{\lambda_1(G_{\b{k}})}
B \log B
\r),
\]
where $G_{\b{k}}$ denotes the lattice
\[
G_{\b{k}}
:=\Big\{(s,t)\in \Z^2:s\equiv \xi_i t \md{k_i}
\
\text{for all}
\
i=1,\ldots,n\Big\}.\]
\end{lemma}
\begin{proof} 
By the Chinese remainder theorem 
there exists an integer $\xi$ such that
$q_i|\xi-\xi_i$
for all $i=1,\ldots,n$.
This shows that 
$(\xi,1)$
is on the lattice
$G_{\b{q}}$,
which is therefore
primitive.
Observing that the condition
$s\equiv \xi t\md{q_1\cdots q_n}$
is equivalent to
$(s,t)$ lying within
$\{\b{A}\b{y}:\b{y}\in \Z^2\}$
with
\[
\b{A}=\left( \begin{array}{cc}
q_1\cdots q_n & \xi  \\
0 & 1 \end{array} \right),
\]
allows us to
use part $(2)$ of
Lemma~\ref{lattice 2}
to evaluate 
$\c{L}_{\bb{\psi}}\!\l(\b{k},\boldsymbol{\xi}\r)$.
The main term will be 
\[
\frac{
\omega_{\bb{\psi}}\!\l(\b{k}\r)
}{\zeta(2)\det(G_{\b{q}})W^2}
\prod_{p | \det(G_{\b{q}})} 
\l(1+\frac{1}{p}\r)^{-1}
\prod_{p | W} 
\l(1-\frac{1}{p^2}\r)^{-1}
,\]
and once we define
$c_0:=W^{-2}\prod_{p \nmid  W} \l(1-\frac{1}{p^2}\r)$
and use that 
$\det(G_{\b{q}})=q_1\cdots q_n$,
we see that it equals
$
c_0 
\omega_{\bb{\psi}}(\b{k})
\frac{f_0(q_1)}{q_1} 
\cdots
\frac{f_0(q_n)}{q_n}
$.
To show that this is the main term stated in our lemma, observe that 
the integers $d_i$ and $m_i$ are squarefree,
and hence
$q_i=m_i^2k_id_i/\gcd(d_i,k_im_i)$,
which implies that
$f_0(q_i)=f_0(m_ik_id_i)$.

By part $(2)$ of
Lemma~\ref{lattice 2}
we see that the error term
in the present lemma is
\[
\ll \frac{\tau(q_1 \cdots q_n)}{\lambda_1(G_\b{q})}
\l(\|\c{R}_{\bb{\psi}}\!\l(\b{k}\r)\|_\infty 
+
\frac{\log\l(1+\|\c{R}_{\bb{\psi}}\!\l(\b{k}\r)\|_\infty\r)}{W\partial(\c{R}_{\bb{\psi}}\!\l(\b{k}\r))^{-1}}
+
\frac{\vol(\c{R}_{\bb{\psi}}\!\l(\b{k}\r))}{W^2(1+\|\c{R}_{\bb{\psi}}\!\l(\b{k}\r)\|_\infty)}
\r),\]
with an absolute implied constant.
The inequality 
$
\vol(\c{R}_{\bb{\psi}}\!\l(\b{k}\r))/(1+\|\c{R}_{\bb{\psi}}\!\l(\b{k}\r)\|_\infty)
\ll  B
$
is valid when $\c{R}_{\bb{\psi}}\!\l(\b{k}\r)$ is empty. 
In the opposite case,
there exists some $\b{x} \in 
\c{R}_{\bb{\psi}}\!\l(\b{k}\r)
\subset
B\c{R}$, which must necessarily be non-zero,
and hence
$\|\b{x}\|_\infty
\geq B
\inf\{\|\b{y}\|_{\infty}:\b{y} \in \c{R}\}
\gg  B$.
This implies that 
$\|\c{R}_{\bb{\psi}}\!\l(\b{k}\r)\|_\infty \gg  B$,
and therefore 
$\c{R}_{\bb{\psi}}(\b{k})\subset B\c{R}$
confirms the validity of
\[
\frac{\vol(\c{R}_{\bb{\psi}}\!\l(\b{k}\r))}{1+\|\c{R}_{\bb{\psi}}\!\l(\b{k}\r)\|_\infty}
\leq 
\frac{\vol(B\c{R})}{1+\|\c{R}_{\bb{\psi}}\!\l(\b{k}\r)\|_\infty}
\ll B.
\] 
The region $\c{R}_{\bb{\psi}}(\b{k})$
is a subset of the box
$B\c{R}$
which is cut out by at most 
$6$ planar curves, each of degree at most $3$, namely 
the curves given by
$|\Delta_i(s,t)|=
\psi_i k_i e_i Q_iW_i.$
This fact reveals that 
$
\partial(
\c{R}_{\bb{\psi}}(\b{k})
) \ll \partial\!\l(B\c{R}\r)\ll B$
and hence 
\[\|\c{R}_{\bb{\psi}}\!\l(\b{k}\r)\|_\infty 
+
\frac{\partial(\c{R}_{\bb{\psi}}\!\l(\b{k}\r))}{W}\log\l(1+\|\c{R}_{\bb{\psi}}\!\l(\b{k}\r)\|_\infty\r)
\ll
B \log B.\]
To finish 
our proof we
observe that 
$\lambda_1(G_\b{k})\geq \lambda_1(G_\b{q})$
can be inferred from $k_i|q_i$.
\end{proof}
Define for all
$(\b{d},\b{m}) \in \c{D} \times \c{M}$,
$\bb{\psi} \in \{0,1\}^n$,
the squarefree integers 
\[
d_i':=\frac{d_i}{\gcd(d_i,m_i)}, 
m_i':=\frac{m_i}{\gcd(d_i,m_i)},
d_i'':=\prod_{\substack{p|d_i \\ p\nmid m_ik_i}}p,
m_i'':=\prod_{\substack{p|m_i\\p\nmid k_i}}p
\] 
and let
\beq{main t}{
T_{\b{d},\bb{\psi}}(B;\b{m})
:=
\sum_{\substack{\b{k} \in \c{K} \\
(1-\psi_i)k_i\neq Q_i
}} 
\hspace{-0,4cm} 
\omega_{\bb{\psi}}\!\l(\b{k}\r)
\prod_{i=1}^n 
\frac{\varrho_i(k_i)
}{k_i}
f_0(k_i,d_i m_i)
\gcd(k_i,d_i')
\tau_i(d_i''m_i'')
.}
The proof of the following result is inspired by an argument of Daniel~\cite[Lem. 3.2]{stephan}.
\begin{lemma}
\lab{prepar+}
For each
$(\b{d},\b{m}) \in \c{D} \times \c{M}$
and $\varepsilon>0$,
we have 
\[S_{\b{d}}(B;\b{m})=c_0
\l(\prod_{i=1}^n
\frac{f_0(d_im_i)}{d_im_i^2}
\gcd(d_i,m_i)
\r)
\l(\sum_{\bb{\psi} \in \{0,1\}^n}
T_{\b{d},\bb{\psi}}(B;\b{m})\r)
+O_{\varepsilon}\!\l(B^{\frac{7}{4}+\varepsilon}\r)
.\] 
\end{lemma}
\begin{proof}
Using~\eqref{bi1231}
in conjunction with
$\gcd(m^2_ik_i,d_i'')=1$,
allows us to procure the validity of
\[
\sum_{\substack{\xi_i\md{q_i}\\
\Delta_i(\xi_i,1)\equiv 0 \md{q_i} 
}}
\hspace{-0,5cm}
\l(\frac{F_i(\xi_i,1)}{k_i}\r)
=
\varrho_i(k_i)
\tau_i(d_i''m_i'')
.\]
The error term in our lemma is of order
$\ll_\varepsilon
B^{1+\varepsilon}
\sum_{k_i \leq Q_i}
\lambda_1(G_\b{k})^{-1}$,
as one can deduce upon congregation of
Lemmas~\ref{miles},~\ref{laatt}
and the estimate~\eqref{national acrobat}.
The fact that the integers 
$k_i$ are coprime in pairs
shows that the lattice
$G_{\b{k}}$ has determinant $k_1 \cdots k_n$.
Therefore,
by
Minkowski's theorem,
its first successive minimum satisfies
\[
\lambda_1(G_\b{k})\ll
(k_1\cdots k_n)^{1/2}
\ll B^{3/4},\]
from which we immediately infer that
the sum over $k_i$ is
\[\ll
\sum_{\substack{\b{v} \in \Z^2\\
\Delta(\b{v}) \neq 0\\
 \|\b{v}\| \ll B^{3/4}}} 
\frac{1}{\|\b{v}\|}
\prod_{i=1}^n
\tau(|\Delta_i(\b{v})|)+
\sum_{\substack{1\le i \le n\\ \deg(\Delta_i)=1}}\!Q_i
.\]
The last term is the contribution
of the primitive
(due to the minimality property of the first successive minimum) non-zero vectors $\b{v}$ 
that are zeros of some $\Delta_i$. 
The number of such vectors is twice the number of linear forms dividing $\Delta$, if any.
Observe that if $\Delta_j$ is linear then the available values $k_i\leq Q_i,i\neq j$,  
with
$\b{v}=\lambda_1(G_{k_1,\ldots,k_n})$ are $O(1)$ in cardinality,
since 
$k_i|\Delta_i(\b{v})$ and $\Delta_i(\b{v})\neq 0$ assumes
finitely many values.
Our proof is now
brought into conclusion
upon deploying the estimate
$\smash
\sum_{0<\|\b{v}\|\leq z}\|\b{v}\|^{-1}\ll z$.
\end{proof}
\subsection{Main term}
\lab{archgoat}
Our aim in this section is to estimate the average $T_{\b{d},\bb{\psi}}(B;\b{m})$.
In contrast to all prior treatments of similar sums,
we provide a uniform treatment for all $\bb{\psi} \in \{0,1\}^n$.   
 
Recall the definition of 
$Q_i$ 
and 
$\omega_{\bb{\psi}}$
provided
in~\eqref{def:qi}
and~\eqref{def:omom}
respectively.
\begin{lemma}
\lab{partials}
\vskip 0.2cm 
$(1)$
For all
$\bb{\psi} \in \{0,1\}^n$
we have
$\omega_{\bb{\psi}}\!\l(1,\ldots,1\r)=B^2\vol(\c{R})+O(B^{3/2}\|\b{m}\|)$.
$(2)$
For all $i$ with $\psi_i\neq 0$ and each $v_i\geq 1$ we have
\[
|\omega_{\bb{\psi}}(\b{v}+\b{e_i})-\omega_{\bb{\psi}}(\b{v})|
 \ll 
B
\|\b{m}\|
\log B
\begin{cases} 
B^{1/2}  &\mbox{if } \deg(\Delta_i)=1,\\ 
v_i^{-1+2/\deg(\Delta_i)} &\mbox{otherwise.}
\end{cases}
\]
$(3)$
The  support of $\omega_{\bb{\psi}}$ is contained in
$
\Big\{\b{v} \in \l(\R_{\geq 1}\r)^n: 1\leq v_i\leq Q_i
\ \text{whenever} \
i \in \mathrm{supp}(\bb{\psi})\Big\}
$.
$(4)$
For all $\b{v}$ we have
$\omega_{\bb{\psi}}(\b{v})\ll  B^2$.
\end{lemma}
\begin{proof}$(1)$
Our statement is obvious if $\bb{\psi}=\b{0}$
and if $\bb{\psi}\neq \b{0}$
we shall show that
for all $i$ in the support of $\bb{\psi}$
we have,
upon denoting
$z_i=m_iW_i^{1/2}\c{A}_i^{1/2}B^{\deg(\Delta_i)/2}$,
that
\beq{2sh}{
\vol\l((s,t)\in B\c{R}:|\Delta_i(s,t)|<z_i\r)
\ll m_i B^{3/2}.}
Our claim
will then follow from
\[
\Big|
\omega_{\bb{\psi}}\!\l(1,\ldots,1\r)-B^2\vol(\c{R})
\Big|
\leq
\sum_{i \in \mathrm{supp}(\bb{\psi})} 
\vol\l((s,t)\in B\c{R}:|\Delta_i(s,t)|<z_i
\r)
.\]
If $\deg(\Delta_i)\geq 3$,
then
$\beta_i=\vol\l((s,t)\in \R^2:|\Delta_i(s,t)|<1\r)<\infty$,
and therefore 
\[
\vol\l((s,t)\in B\c{R}:|\Delta_i(s,t)|<z_i\r)
\leq \beta_i z_i^{2/\deg(\Delta_i)}\ll m_i B,\]
we can also treat similarly the case when 
$\Delta_i$ is quadratic and irreducible in $\R[s,t]$.
If $\Delta_i$ is linear, then a linear change of variables 
shows that
\[
\vol\l((s,t)\in B\c{R}:|\Delta_i(s,t)|<z_i\r)
\ll Bz_i \ll m_i B^{3/2},\]
which is sufficient.
The last remaining case is when
$\Delta_i$ is quadratic and splits over $\R$, in this case
it can be transformed into $st$. Denoting by $\c{R}'$ 
the image of $\c{R}$ under this transformation,
we observe that 
the volume of
$
\{(s,t)\in B\c{R}':|st|<z_i
\}$
is
\[
\ll
\vol\l((s,t)\in B\c{R}':|s|<1\r)
+|z_i|
\int_{1\le |s|\ll B}\frac{\mathrm{d}s}{|s|}
\ll
B+m_i B \log B
,\]
which is sufficient for our proof.
\newline
$(2)$
In the case that
$\deg(\Delta_i)\geq 3$, we can
follow the notation of the proof of part $(1)$ to obtain
\[
\vol\l(
(s,t) \in B\mathcal{R}:
v_i\le  \frac{|\Delta_i(s,t)|}{z_i}< v_i+1
\r)
\leq \beta_i
z_i^{\frac{2}{\deg(\Delta_i)}}
\l((v_i+1)^{\frac{2}{\deg(\Delta_i)}}-v_i^{\frac{2}{\deg(\Delta_i)}}\r),
\]
which by the mean value theorem is
$\ll B m_i^{2/\deg(\Delta_i)}
v_i^{1-2/\deg(\Delta_i)}$.
The remaining cases $\deg(\Delta_i)=1$ or $2$, are treated in a similar fashion.
\newline
$(3)$ It flows directly from
the fact that $|\Delta_i(s,t)| \leq 
\c{A}_i B^{\deg(\Delta_i)}$ for all $(s,t) \in B\c{R}$.
\newline
$(4)$ This part is a direct consequence of the inclusion 
$\c{R}_{\bb{\psi}}\!\l(\b{v}\r)\subset B\c{R}$.
\end{proof} 
\begin{lemma}
\lab{triumphat}
Assume that
$h \in \c{U}$,
$b$ is coprime to $W$
and
define
$\gamma_i=2/(4+\deg(\Delta_i)).$
Then there exist 
$
W_1 \in W\N,
c \in \R_{>0}
$
and
$h_1 \in \c{U}$,
such that 
for all
$0<\varepsilon <\gamma_i$ and
$x>0$, we have
\[
\sum_{\substack{k \leq x\\ \gcd(k,bW_1)=1}}
\hspace{-0,5cm}
\frac{\varrho_i(k)}{k}
h(k)
=c
h_1(b)
+O_{\varepsilon}\!\l(
\frac
{(bx)^\varepsilon}
{x^{\gamma_i}}
\r),
\]
where the implied constant is independent of $b$ and $x$.
\end{lemma}
\begin{proof}
We shall begin by proving our claim in the case $x\geq 1$.
Define the arithmetic function 
$g_k:\N\to\R$ by
\[
g_k:=
\begin{cases}
h(k) \varrho_i(k)  &\mbox{if } \gcd(k,bW)=1,\\ 
0 &\mbox{otherwise }
\end{cases}\]
and denote its Dirichlet series by $F_g(s)$.
Multiplying $W$ by sufficiently large primes 
produces a multiple $W_1$ of $W$ and then
by~\eqref{bwv 244} we have that for all 
$z \in \mathbb{C}$ with $|z|\leq 2^{-1/2}$ 
and
$p\nmid W_1$
one has
$
R_i(p,z)
\ll |z|
$.
The fact that $h\in \c{U}$ shows that 
in the same range for $z$ and for all primes $p\nmid bW_1$,
one has 
\[\l(\sum_{m=1}^\infty
g_{p^m}z^m\r)\l(1+R_i(p,z))\r)^{-1}
=\l(1+h(p)R_i(p,z)\r)\l(1+R_i(p,z)\r)^{-1}=1+O\l(\frac{|z|}{p}\r)
,\]
which reveals that the product
\[\Phi_b(s):=\prod_{p\nmid bW_1}
\frac{\sum_{m=1}^\infty
g_{p^m}p^{-ms}}{1+R_i(p,p^{-s})},
\]
defined for all $b \in \N$ with $\gcd(b,W)=1$,
converges absolutely in the region $\Re(s)>\frac{1}{2}$.  
The bound
\[\frac{\sum_{m=1}^\infty
g_{p^m}p^{-ms}}{1+R_i(p,p^{-s})}
=1+O(p^{-3/2})
\] then proves that $\Phi_b(s)$ has no zero in the same region.

Recalling definition~\eqref{alladin sane}
enables us to see that
the following factorisation is valid in the region 
$\Re(s)>1$,
\[
F_g(s)
=
\Phi_b(s)
P_i(s)
\prod_{p|b}\l(1+R_i(p,p^{-s})\r)^{-1}
,\]
and therefore by analytic continuation it is also valid for 
$\Re(s)>\frac{1}{2}$. This shows that $F_g$ may be extended to 
$\Re(s)>\frac{1}{2}$
and that, owing to~\eqref{krateio}, it satisfies the bound
\beq{bbb}{F_g(s)
\ll_\varepsilon
b^\varepsilon
(1+|\Im(s)|)^{\frac{\deg(\Delta_i)}{2}(1-\Re(s))+\varepsilon},
\frac{1}{2}<\Re(s)<1.}
Enlarging $W$
ensures that the quantities
\[c=\prod_{p\nmid W}\l(1+h(p)R_i(p,p^{-1})\r),
h_1(n)=\prod_{p|n}\l(1+h(p)R_i(p,p^{-1})\r)^{-1} 
\]
satisfy $c>0 $ and $h_1 \in \c{U}$,
once the factorisation
$c=\Phi_1(1)P_i(1)$ and~\eqref{ippo} have been combined.
Letting $G(x):=\sum_{k\leq x}g_k$,
a computation 
involving Euler products
reveals
that 
the constant
$F_g(1)=\sum_{k=1}^{\infty}g_k/k=\int_1^\infty G(u)/u^{2}\mathrm{d}u$
equals
$ch_1(b)$.

Our aim for the rest of this proof is to show that $
G(x)\ll_\varepsilon b^\varepsilon x^{1-\gamma_i+\varepsilon}$, 
by partial summation this would be sufficient for our lemma.
The fact that $F_g(s)$ converges for $\Re(s)>\frac{1}{2}$
immediately proves such an estimate, save for the dependence on
$b$.
However,
Making this dependence explicit is
the main point of our lemma;
this requirement makes using contour integration indispensable.

Letting $x\geq 1$ be
a half-integer and
using~\cite[Cor.5.3]{mova}
with
$\sigma_0=1+1/\log x$
and
$T=x^{\gamma_i}$,
gives
\[G(x)=\frac{1}{2\pi i}\int_{\sigma_0-iT}^{\sigma_0+iT}F_g(s)
\frac{x^s}{s}
\mathrm{d}s
+O\l(R_0(\max_{k\le 2x}|g_k|)+\frac{x}{T}\sum_{k=1}^\infty\frac{|g_k|}{k^{\sigma_0}}\r)
,\]
where
\[R_0:=\sum_{\substack{\frac{x}{2}<n<2x \\n \neq x}}\min\l(1,\frac{x}{T|x-n|}\r)\]
satisfies 
$R_0\ll 1+\frac{x \log x}{T}$, as shown for example in~\cite[p.g. 180]{mova}.
By~\eqref{pointw} and~\eqref{national acrobat} we deduce that there exists $A>0$ such that
\[|\varrho_i(k)h(k)|\ll
\tau(k)^A.\]
This shows that 
\[
\sum_{k=1}^\infty\frac{|g_k|}{k^{\sigma_0}}
\ll
\zeta(\sigma_0)^{2^A},
\]
and therefore
the bound
$\zeta(1+\delta)\ll \delta^{-1}$,
valid for all $\delta>0$,
yields
\beq{int1}{G(x)=\frac{1}{2\pi i}\int_{\sigma_0-iT}^{\sigma_0+iT}F_g(s)
\frac{x^s}{s}
\mathrm{d}s
+O_\varepsilon\!\l(\frac{x^{1+\varepsilon}}{T}\r)
.}
Taking advantage of the fact that $F_g$ has no poles in the rectangle 
enclosed by the points
$\sigma_0\pm iT$ and $\frac{1}{2}+\varepsilon\pm iT$
we may replace the vertical line of integration in~\eqref{int1}
by two horizontal and a vertical line.
An easy computation that makes use of~\eqref{bbb}
allows us to provide satisfactory
upper
bounds for the contribution coming from the three remaining
segments, thus finalising the proof in the range
$x\geq 1$. In the remaining case
$0<x<1$ we notice that the sum over $k$ in the statement of our lemma contains no terms
and therefore the estimates
$c\ll 1$ and $h(b)\ll_\varepsilon b^\varepsilon$
prove that our claim remains valid in the region $x<1$.
\end{proof}
Our next result regards averages of certain
non-multiplicative functions that will appear in the
forthcoming estimation of
$T_{\b{d},\bb{\psi}}(B;\b{m})$.
Recall definition~\eqref{bowie_low}.
\begin{lemma}
\lab{awky}
Suppose we are given a function $h \in \c{U}$
and integers $d,m$ and $k_0$
satisfying
\[1=\gcd(k_0,dm)=\gcd(k_0dm,W)=\mu(d)^2=\mu(m)^2\]
and define the integers
\[d'=\prod_{\substack{p|d\\p\nmid m}}p,
m'=\prod_{\substack{p|m\\p\nmid d}}p
.\]
There exist $h_j \in \c{U}$ and a constant $c>0$ 
such that for all $0<\varepsilon<\gamma_i,x\geq 1,$ 
we have
\begin{align*}
&\sum_{\substack{k \leq x \\ \gcd(k,Wk_0)=1}}
\frac{\varrho_i(k)}{k}
h(k,dm')
\tau_i\!\l(\frac{dm'}{\gcd(k,dm')}\r)
\gcd(k,d')=\\
&\ \ \
c h_1(k_0)
h_2(d)
h_3(m')
\sigma_i(d)
\tau^\sharp_i(m')
+O_\varepsilon\l(
\frac{(dmk_0x)^\varepsilon}{x^{\gamma_i}}
\r)
.\end{align*}
\end{lemma}
\begin{proof}
Writing $\delta_1=\gcd(k,d),
\delta_2=\gcd(k,m')$
transforms our sum into
\[\sum_{\substack{\delta_1|d\\
\delta_2|m'}}
\tau_i\!\l(\frac{d}{\delta_1}\r)
\tau_i\!\l(\frac{m'}{\delta_2}\r)
\delta_2^{-1}
\hspace{-0,3cm}
\sum_{\substack{ l \leq x/(\delta_1\delta_2) \\ \gcd(l,Wk_0dm'/(\delta_1\delta_2))=1}}
\hspace{-0,3cm}
\frac{\varrho_i(\delta_1\delta_2l)}{l}
h(l,\delta_1\delta_2) 
.\]
For each $\delta \in \N$ we define the 
set $\mathbb{M}(\delta)\subset \N$ comprising of all positive integers 
whose
prime factors divide $\delta$ 
and we allow 
$1\in \mathbb{M}(\delta)$. Then
each natural $l$ can be decomposed uniquely as
$l=l_1l_2l'$ where 
$l_j \in \mathbb{M}(\delta_j)$ and 
$l'$ is coprime to $\delta_1\delta_2$, thus leading to
the equality of the sum
over $l$ with
\[
\sum_{l_j \in
\mathbb{M}(\delta_j)}
\frac{\varrho_i(\delta_1l_1)}{l_1}
\frac{\varrho_i(\delta_2l_2)}{l_2}
\sum_{\substack{ l' \leq x/(l_1 l_2) \\ \gcd(l',Wk_0dm')=1}}
\hspace{-0,3cm}
\frac{\varrho_i(l')}{l'}
h(l') 
.\]
Lemma~\ref{triumphat}
provides a function $h_1\in \c{U}$,
a constant $c>0$
and a multiple of $W$, which we identify here with $W$ itself,
such that the sum over $l'$ equals
\[c h_1\!(k_0dm')
+O_\varepsilon\l(\frac{(k_0dmx)^\varepsilon}{x^{\gamma_i}}
(l_1 l_2)^{\gamma_i}\r)
.\]
The estimate~\eqref{national acrobat}
reveals the validity of
$\varrho_i(\delta_jl_j)\ll_\varepsilon (\delta_jl_j)^\varepsilon
\leq (dml_j)^\varepsilon$. This shows that 
\[
\sum_{l_j \in \mathbb{M}(\delta_j)}
\frac{\varrho_i(\delta_1l_1)}{l_1}
\frac{\varrho_i(\delta_2l_2)}{l_2}
\ll_\varepsilon
(dm)^{2\varepsilon}
\prod_{p|\delta_1\delta_2}
\frac{1}{1-p^{\varepsilon-1}}
,\]
and the proof of our lemma follows upon noting that 
\[
\sum_{l_j \in \mathbb{M}(\delta_j)}
\frac{\varrho_i(\delta_jl_j)}{l_j}
=\prod_{p|\delta_j}
\l(\sum_{k=0}^\infty
\frac{\varrho_i(p^{k+1})}{p^{k}}
\r),
j=1,2.\]
\end{proof}
Recall~\eqref{main t}
and define
\beq{main ct}{
T^\natural_{\b{d},\bb{\psi}}(B;\b{m}):=
\sum_{\b{k} \in \c{K}} 
\omega_{\bb{\psi}}\!\l(\b{k}\r)
\prod_{i=1}^n 
\frac{\varrho_i(k_i)
}{k_i}
f_0(k_i,d_i m_i)
\gcd(k_i,d_i')
\tau_i(d_i''m_i'')
.}
One has
\[
T^\natural_{\b{d},\bb{\psi}}(B;\b{m})
-
T_{\b{d},\bb{\psi}}(B;\b{m})
=
\Osum_{\!\!\b{k} \in \c{K}} 
\omega_{\bb{\psi}}\!\l(\b{k}\r)
\prod_{i=1}^n 
\frac{\varrho_i(k_i)
}{k_i}
f_0(k_i,d_i m_i)
\gcd(k_i,d_i')
\tau_i(d_i''m_i'')
,\]
where
$\Osum$ is over
$\b{k}$ with
 $k_i=Q_i$ if $\psi_i=0$
and $1\leq k_i\leq Q_i$ if $\psi_i=1$.
The third part of Lemma~\ref{partials}
in conjunction
with~\eqref{national acrobat}
then leads to
\beq{t:2}{
T^\natural_{\b{d},\bb{\psi}}(B;\b{m})
-
T_{\b{d},\bb{\psi}}(B;\b{m})
\ll_{\varepsilon}\!\|\bb{\psi}\|
B^{\frac{3}{2}+\varepsilon}
\|\b{d}\|^n
.} 

A gambit is on offer in the proof of the ending lemma of the present section;
we use iterated partial summation 
instead of the more natural 
multidimensional partial summation
to deal with the factor 
$\omega_{\bb{\psi}}\!\l(\b{k}\r)$
present in $T_{\b{d},\bb{\psi}}(B;\b{m})$. 
While this approach complicates the argument,
it has the advantage of providing a uniform treatment for 
all vectors $\boldsymbol{\psi}$ and forms $\Delta_i$.

\begin{lemma}
\lab{lemmon}
There exist
$c_1>0$ 
and multiplicative functions 
$\mathds{a}_{j,i}
\in \c{U}$,
such that 
\[
T^\natural_{\b{d},\bb{\psi}}(B;\b{m})=
c_1
B^2
\l(\prod_{i=1}^n
\mathds{a}_{1,i}(d_i)
\sigma_i\!\l(d_i\r) 
\tau_i^\sharp\!\l(m_i'\r) 
\mathds{a}_{2,i}(m_i,d_i) 
\r)
+O\!\l(
\|\b{m}\|
B^{\frac{19}{20}}
\r)
,\]
where the implied constant is independent of $B,\b{d},\b{m}$ and $\bb{\psi}$.
\lab{prepar}
\end{lemma}
\begin{proof}
Letting 
$
\c{K}_i=
\left\{
k_i\in \N\cap [1,Q_i]:
\gcd(k_i,W\prod_{j<i}d_jm_jk_j)=1
\right\}
$
allows us to express 
$T^\natural_{\b{d},\bb{\psi}}(B;\b{m}) $ in the form 
\beq{erfo}{\sum_{k_1 \in \c{K}_1}
\frac{\varrho_1(k_1) f_0(k_1,d_1 m_1)
\tau_1(d_1''m_1'')
}{\gcd(k_1,d_1')^{-1}k_1}
\ldots
\sum_{k_n \in \c{K}_n}
\frac{\varrho_n(k_n) f_0(k_n,d_nm_n)
\tau_n(d_n''m_n'')
}{\gcd(k_n,d_n')^{-1}k_n}
\omega_{\bb{\psi}}\!\l(\b{k}\r)
.}
Let us note that
for all
$\mathds{a} \in \c{U}$ and
$d,m,k \in \Z$,
the identity
\beq{eq:anna mirage}{
\mathds{a}(dmk)
=
\mathds{a}(d)
\mathds{a}(m,d)
\mathds{a}(k,dm)
}
holds due to~\eqref{genesis}.
We shall make repeated use of
Lemma~\ref{triumphat}
to study the
sum over $k_n,k_{n-1},\ldots,k_1$
and at every step of this process,
products of multiplicative functions in $\c{U}$ evaluated at 
several combinations of products and quotients of the integers $d_i,m_i,k_i$ and $\gcd(d_i,m_i)$
will enter the stage.~
Therefore,
using
the coprimality conditions imprinted in the definitions of $\c{D},\c{M},\c{K}$
and~\eqref{all+},
we shall always use of~\eqref{eq:anna mirage} and the group law in $\c{U}$ to rewrite
these terms in the succinct form
\beq{star}{ 
\l(\prod_{i=1}^n \mathds{a}_{1,i}(d_i)\r)\l(
\prod_{i=1}^n \mathds{a}_{2,i}(m_i,d_i) \r)\l(
\prod_{i=1}^n \mathds{a}_{3,i}(k_i,d_im_i) \r)
,}
for some $\mathds{a}_{i,j} \in \c{U}$.
While
at each successive
stage
the values of the functions $\mathds{a}_{i,j}$ will vary,
this will, however,
leave
unspoilt the validity of our lemma.

We begin by applying Lemma~\ref{awky} 
for
$i=n$,
$d=d_n$,
$m=m_n$
and
\[ 
k_0=\prod_{i=1}^{n-1}
d_im_ik_i
.\]
Denoting for all $t\in \R \cap [1,2Q_n]$,
\[S(t):=
\hspace{-0,4cm}
\sum_{\substack{k_n\leq t\\ \gcd(k_n,Wk_0)=1}}
\frac{\varrho_n(k_n) f_0(k_n,d_nm_n)
\tau_n(d_n''m_n'')
}{\gcd(k_n,d_n')^{-1}k_n},
\]
and taking under consideration
the inequalities
$d_i,m_i,k_i\ll B^n$ 
we deduce that
\[S(t)=
c_n  \sigma_n\!\l(d_n\r) 
\tau_n^\sharp\!\l(m_n'\r) 
\gamma
+O_\varepsilon\!
\l(
B^\varepsilon
t^{-\gamma_n}\r)
,\]
where
\[\gamma=h_{1}\!\l(\prod_{j\neq n}d_jm_jk_j\r)
h_2(d_n)
h_3(m_n,d_n)
\]
for a constant $c_n>0$
and functions $h_i \in \c{U}$, all of which are
independent of $B,d_i,m_i$ and $k_i$. 
The integers $d_im_ik_i$ are coprime in pairs, and hence letting
\[
\mathds{a}_{1,i}=
\mathds{a}_{2,i}=
\mathds{a}_{3,i}=
h_1, i\neq n
\ \ \text{and} \ \
\mathds{a}_{1,n}=h_2,
\mathds{a}_{2,n}=h_3,
\mathds{a}_{3,n}=1,
\]
allows us to see that
$\gamma$ is of the shape~\eqref{star}.

Define for fixed $k_1,\ldots,k_{n-1}$ the function
$\omega_n(x):=\omega_{\bb{\psi}}(k_1,\ldots,k_{n-1},x)$
and denote by $S_\omega(t)$
the sum one obtains by replacing 
$\varrho_n(k_n)$ in the definition of $S(t)$ by
$\varrho_n(k_n)\omega_n(k_n)$.
If $\psi_n=0$ then the definition of $\omega_{\boldsymbol{\psi}}(\b{v})$
implies 
\[S_{\omega}(Q_n)=
\omega_{(\psi_1,\ldots,\psi_{n-1},0)}(k_1,\ldots,k_{n-1},1) 
S(Q_n),\]
while 
if $\psi_n=1$ then
letting
$z:=[Q_n]+1$
and observing that 
$\omega_n(z)=0$, by
part $(3)$ of Lemma~\ref{partials}
and
the discrete version of 
partial summation
we are provided with
\[
S_\omega(Q_n)=
\sum_{\substack{t \in \N \\ 1 \le t \le z-1}}
S(t)(\omega_n(t)-\omega_n(t+1))
,\]
which equals
\[
\omega_n(1)
c_n  \sigma_n\!\l(d_n\r) 
\tau_n^\sharp\!\l(m_n'\r) 
\gamma 
+O_\varepsilon\!\!\l(
B^\varepsilon
\!
\sum_{\substack{t \in \N \\ 1 \le t \le z-1}}
\frac{|\omega_n(t+1)-\omega_n(t)|}{
t^{\gamma_n}}\r).\]
Employing 
part $(2)$ of Lemma~\eqref{partials}
for
$i=n,\b{v}=(k_1,\ldots,k_{n-1},t)$
and
letting
\[\omega_{n-1}(t)=\omega_{(\psi_1,\ldots,\psi_{n-1},0)}(k_1,\ldots,k_{n-2},t,1),\]
a simple computation reveals that
\[
S_\omega(Q_n)=
\omega_{n-1}(k_{n-1})
c_n  \sigma_n\!\l(d_n\r) 
\tau_n^\sharp\!\l(m_n'\r) 
\gamma 
+O_\varepsilon\!
\l(
\|\b{m}\|
B^{1+\frac{4}{4+\deg(\Delta_n)}+\varepsilon}
\r),\]
regardless of the value of
$\psi_n$.
Letting 
\[
\gamma'=
c_n  \sigma_n\!\l(d_n\r) 
\tau_n^\sharp\!\l(m_n'\r) 
\l(\prod_{i=1}^n \mathds{a}_{1,i}(d_i)\r)\l(
\prod_{i=1}^n \mathds{a}_{2,i}(m_i,d_i) \r) 
\]
and noting that $S_\omega(Q_n)$ is the sum over $k_n$ in~\eqref{erfo},
we deduce 
that,
up to an admissible error term,
$T^\natural_{\b{d},\bb{\psi}}(B;\b{m})$ equals 
\begin{align*}
&\gamma'
\sum_{k_1 \in \c{K}_1}
\frac{\varrho_1(k_1)}{k_1}
\tau_1(d_1''m_1'')
\gcd(k_1,d_1')
(\mathds{a}_0
\mathds{a}_{3,1})(k_1,d_1 m_1)
\ldots
\sum_{k_{n-1} \in \c{K}_{n-1}}
\frac{\varrho_{n-1}(k_{n-1})}{k_{n-1}}
\times \\ &
\times \tau_1(d_{n-1}''m_{n-1}'') 
\gcd(k_{n-1},d_{n-1}') (\mathds{a}_0
\mathds{a}_{3,n-1})(k_{n-1},d_{n-1} m_{n-1})
\omega_{n-1}(k_{n-1})
.\end{align*} 
Repeating this process inductively to evaluate the sum over $k_{n-1}, \ldots, k_1$,
yields the desired result
upon using part $(1)$ of Lemma~\ref{partials}.
\end{proof}
Perusing~\eqref{t:2}, Lemmas~\ref{prepar+},~\ref{lemmon},
and letting
\[
c=2^nc_0c_1,
g_i=(\mathds{a}_0 \mathds{a}_{1,i})^{-1},
h_i=(\mathds{a}_0
\mathds{a}_{2,i})^{-1}
,\]
allows us to
achieve
the anticipated catharsis
of the verification of Theorem~\ref{b evans}.  
\section{Employing the Rosser--Iwaniec sieve}
\label{combo mac iwa}
We shall make use of the \textit{Fundamental lemma of sieve theory};
we choose to employ the version of the lemma supplied
in~\cite[Lem. 6.3]{iwa}. 
\begin{lemma}[Fundamental lemma of sieve theory]
\label{funda}
Let $\kappa>0$ and $y>1$. There exist two sets of real numbers 
$\Lambda^+=(\lambda_d^+)$ and 
$\Lambda^-=(\lambda_d^-)$
depending only on $\kappa$ and $y$ with the following properties:
\begin{align}
&\lambda_1^\pm=1, \label{643}\\
|&\lambda_d^\pm| \leq 1 \ \text{if} \ 1<d<y, \label{644} \\
&\lambda_d^\pm=0 \ \text{if} \ d\geq y,\label{645}
\end{align}
and for any integer $n>1$,
\beq{646}{\sum_{d|n}\lambda_d^-\leq 0 \leq \sum_{d|n}\lambda_d^+.}
Moreover, for any multiplicative function $g(d)$ with $0\leq g(p)<1$ and satisfying the dimension condition
\beq{647}{\prod_{w\leq p <z}(1-g(p))^{-1}\leq
\l(\frac{\log z}{\log w}\r)^\kappa
\l(1+\frac{K}{\log w}\r)}
for all $2\leq w <z\leq y$ we have 
\beq{648}{\sum_{d|P(z)}\lambda_d^{\pm}g(d)=
\l(1+O\!\l(e^{-1/{\varpi}}\l(1+\frac{K}{\log z}\r)^{10}\r)\r)
\prod_{p<z}(1-g(p)),}
where $P(z)$ denotes the product of all primes $p<z$ and $
{\varpi}
=\log z/\log y$,
the implied constant
depending only on $\kappa$.
\end{lemma} 
Lemma~\ref{19.a} and
Proposition~\ref{r2d}
show that if
$\lambda^-_d$
is chosen with respect to $y=B^{\frac{1}{100n(n+1)}}$
then
\[
N(\pi,B)
\gg
B^2
\sum_{\b{d} \in \c{D}}
\lambda_{d_1\cdots d_n}^-
\prod_{i=1}^n
\frac{\sigma_i(d_i)}{\tau(d_i)d_i}u_i(d_i)
,
\] 
where we have made use of~\eqref{644},\eqref{645} and
$\sum_{d_1\cdots d_n\leq y}\|\b{d}\|^n\ll y^{n+1+\varepsilon}$.
Therefore Theorem~\ref{thm:lower} would
follow from proving that
the sum over $\b{d}$ is $\gg (\log B)^{-n/2}$.
We shall do so by applying Lemma~\ref{funda}
and we begin by verifying its assumptions.

Defining the multiplicative function
\[g(d)=
\frac{\mu(d)^2}{d\tau(d)}
\b{1}_W(d)
\hspace{-0,5cm}
\sum_{\substack{\b{d} \in \N^n \\
d_1\cdots d_n=d \\   
\gcd(d_i,d_j)=1,
i\neq j 
}}
\prod_{i=1}^n
\sigma_i(d_i)u_i(d_i)
,\]
where
$\b{1}_W$ 
is the characteristic functions of the integers coprime to $W$,
allows us to write
\[
\sum_{\b{d} \in \c{D}}
\lambda_{d_1\cdots d_n}^{-}
\prod_{i=1}^n
\frac{\sigma_i(d_i)}{\tau(d_i)d_i}
=\sum_{
d|P
(B^{\varpi/100n(n+1)})
}
\lambda_d^{-}
g(d)
\]
and we shall next show that~\eqref{647} is satisfied with 
$\kappa=n/2$.
We have
$g(p)=0$ for all $p\leq D_0$,
while in the range
$p>D_0$ we have
\[g(p)=\frac{1}{2p}
\sum_{i=1}^n
\sigma_i(p)
u_i(p)
.\]
The estimates 
$u_i(p),\sigma_i(p)\ll 1$ reveal that
$g(p)\ll 1/p$ and hence,
enlarging $D_0$ if necessary,
we obtain $g(p)<1$. 
The fact that $u_i(p)\geq 0$ follows from
$u_i\in \c{U}$, thereby showing that
the inequality 
$g(p)\geq 0$ is a consequence of
$\sigma_i(p)\geq 0$.
Indeed, by~\eqref{bwv 988},
\[
\sigma_i(p)=\sum_{\substack{\xi\md{p}\\ \Delta_i(\xi,1)\equiv 0 \md{p}}}
\l(1+\l(\frac{\delta_i(\xi,1)}{p}\r)\r)
+O\!\l(\frac{1}{p}\r),
\]
which is a non-negative even integer  
up to an error $O(\frac{1}{p})$,
thus
for large $p$ it attains negative values only when
$\varrho_i(p)=-\tau_i(p).$ In this case~\eqref{bwv 988}
reveals that
$\sigma_i(p)$ equals
$\tau_i(p)(p+1)^{-1}$,
which is again non-negative.

Using a Taylor expansion
leads us to
\[\log \prod_{  p <z}(1-g(p))^{-1}
=\sum_{k=1}^{\infty}\frac{1}{k}\sum_{  p <z}g(p)^k
\]
whenever
$z>D_0$.
The estimate
$g(p)\ll 1/p$ shows
that the sum of all terms with $k\geq 2$ equals
$c_0+O(1/z)$ for some constant $c_0$ independent of $z$.
Let us observe that
the sequence
\[a(p)=\sum_{i=1}^n\l(\sigma_i(p)-\tau_i(p)-\varrho_i(p)\r)\] satisfies $a_i(p)\ll 1/p$,
and hence
the remaining sum equals
\[\sum_{ p <z}g(p)
= 
\frac{1}{2}\sum_{i=1}^n 
\l(\sum_{D_0<p<z}\frac{\tau_i(p)}{p}\r)
+
\frac{1}{2}\sum_{i=1}^n 
\l(\sum_{D_0<p<z}\frac{\varrho_i(p)}{p}\r) 
+a+O\!\l(\frac{1}{z}\r),\]
where $a=\sum_{p>D_0}a_p/p$.
By~\eqref{mertenoulis} and~\eqref{mertenoulis'}  
we can therefore deduce that  
\[\sum_{ p <z}g(p)=\frac{n}{2}\log \log z +c_1+O\!\l(\frac{1}{\log z}\r),\] 
thus infering
\[\prod_{p<z}(1-g(p))^{-1}
=e^{c_0+c_1}
\l( \log z \r)^{n/2}
\l(1+O(1/\log z)\r),
\] from which~\eqref{647} follows in the case $w>D_0$.
In the remaining cases
$w\leq D_0<z$ and $z<w \leq D_0$,
the product
$\prod_{ w\leq p <z}(1-g(p))^{-1}$
equals
 $\prod_{1+D_0\leq  p <z}(1-g(p))^{-1}$ 
and
$1$ respectively;
they are
$\leq (\log z/\log w)^{n/2}
\l(1+O(1/\log w)\r)$ in both cases.
This finishes the verification of~\eqref{646},
we have therefore obtained the validity of
\[\sum_{d|P(B^{{\varpi}/100n(n+1)})}\lambda_d^-
g(d)
=\l(1+O\l(e^{-1/\varpi}\l(1+\frac{1}{\varpi\log B}\r)^{10}\r)\r)
\prod_{ p <B^{{\varpi}/100n(n+1)}}(1-g(p))
.\] 
Fixing a suitably
small positive value for $\varpi$ 
ensures that the sum 
behaves asymptotically as
\[
\prod_{ p <B^{{\varpi}/100n(n+1)}}(1-g(p))
\asymp \l(\log B\r)^{-n/2},\]
an estimate
which concludes the proof of Theorem~\ref{thm:lower}.
\bibliographystyle{amsalpha}
\bibliography{neutlast}
\end{document}